\documentclass[12pt,reqno,openbib,runningheads,a4paper]{amsart}%
\usepackage{graphicx}
\usepackage{amssymb}
\usepackage{amsfonts}
\usepackage{amsmath}
\usepackage{graphicx}
\usepackage{lineno}
\usepackage[authoryear]{natbib}
\usepackage[dvipsnames]{xcolor}
\usepackage{epsf}
\usepackage{lscape}
\usepackage[legalpaper,bookmarks=true,colorlinks=true,linkcolor=blue,citecolor=blue]%
{hyperref}%
\providecommand{\U}[1]{\protect\rule{.1in}{.1in}}
\providecommand{\U}[1]{\protect\rule{.1in}{.1in}}
\newtheorem{theorem}{Theorem}[section]

\newtheorem{proposition}{Proposition}[section]

\newtheorem{lemma}{Lemma}[section]

\newtheorem{remark}{Remark}[section]

\makeatletter
\renewcommand{\@biblabel}[1]{}
\makeatother
\begin{document}

\begin{center}
{\Large \textbf{Estimating the mean of a heavy-tailed distribution under
random censoring}}\medskip\medskip

{\large Louiza Soltane, Djamel Meraghni, Abdelhakim Necir}$^{\ast}$\medskip

{\small \textit{Laboratory of Applied Mathematics, Mohamed Khider University,
Biskra, Algeria}}\bigskip\medskip
\end{center}

\noindent\textbf{Abstract}\medskip

\noindent The central limit theorem introduced by Stute [The central limit
theorem under random censorship. Ann. Statist. 1995; 23: 422-439] does not
hold for some class of heavy-tailed distributions. In this paper, we make use
of the extreme value theory to propose an alternative estimating approach of
the mean ensuring the asymptotic normality property. A simulation study is
carried out to evaluate the performance of this estimation procedure.\medskip

\noindent\textbf{Keywords:} Central limit theorem; Empirical process; Hill
estimator; Kaplan-Meier estimator; Random censoring.\medskip

\noindent\textbf{AMS 2010 Subject Classification:} 62P05; 62H20; 91B26; 91B30.

\vfill

\noindent{\small $^{\text{*}}$Corresponding author:
\texttt{necirabdelhakim@yahoo.fr} \newline\noindent\textit{E-mail
addresses:}\newline\texttt{louiza\_stat@yahoo.com} (L.~Soltane)\newline%
\texttt{djmeraghni@yahoo.com} (D.~Meraghni)}

\section{\textbf{Introduction\label{sec1}}}

\noindent Let $X_{1},...,X_{n}$ be $n\geq1$ independent copies of a
non-negative random variable (rv) $X,$ defined over some probability space
$(\Omega,\mathcal{A},\mathbb{P)},$ with absolutely continuous cumulative
distribution function (cdf) $F.$ An independent sequence of independent rv's
$Y_{1},...,Y_{n},$ with absolutely continuous cdf $G,$ censor them to the
right, so that at each stage $j$ we can only observe $Z_{j}:=\min(X_{j}%
,Y_{j})$ and $\delta_{j}:=\mathbf{1}\left\{  X_{j}\leq Y_{j}\right\}  ,$ with
$\mathbf{1}\left\{  \cdot\right\}  $ denoting the indicator function. The rv
$\delta_{j}$ indicates whether or not there has been censorship. Throughout
the paper, we use the notation $\overline{\mathcal{S}}(x):=\mathcal{S}%
(\infty)-\mathcal{S}(x),$ for any $\mathcal{S}.$ If $H$ denotes the cdf of the
observed $Z^{\prime}s,$ then, by the independence of $X_{1}$ and $Y_{1},$ we
have $\overline{H}\left(  z\right)  =\overline{F}\left(  z\right)
\overline{G}\left(  z\right)  .$ In our work, we assume that both $F$ and $G$
are heavy-tailed, this means that there exist to constants $\gamma_{1}>0$ and
$\gamma_{2}>0,$ called tail indices, such that%
\begin{equation}
\underset{z\rightarrow\infty}{\lim}\frac{\overline{F}(xz)}{\overline{F}%
(z)}=x^{-1/\gamma_{1}}\text{ and }\underset{z\rightarrow\infty}{\lim}%
\frac{\overline{G}(xz)}{\overline{G}(z)}=x^{-1/\gamma_{2}}, \label{Condition}%
\end{equation}
for any $x>0.$ Consequently, $H$ is heavy-tailed too, with tail index
$\gamma:=\gamma_{1}\gamma_{2}/(\gamma_{1}+\gamma_{2}).$ The class of
heavy-tailed distribution takes a significant role in extreme value theory. It
includes distributions such as Pareto, Burr, Fr\'{e}chet, $\alpha-$stable
$\left(  0<\alpha<2\right)  $ and log-gamma, known to be appropriate models
for fitting large insurance claims, log-returns, large fluctuations of prices,
etc. (see, e.g., \citeauthor{R07}, \citeyear{R07}). Examples of censored data
with apparent heavy tails can be found in \cite{GN11}. The nonparametric
maximum likelihood estimator of $F$ is given by \cite{KM58} as the product
limit estimator%
\[
F_{n}(x):=\left\{
\begin{tabular}
[c]{ll}%
$1-%
{\displaystyle\prod\limits_{Z_{j:n}\leq x}}
\left(  \dfrac{n-j}{n-j+1}\right)  ^{\delta_{\left[  j:n\right]  }}$ & for
$x<Z_{n:n}$\\
$1$ & for $x\geq Z_{n:n},$%
\end{tabular}
\ \ \ \ \ \ \ \ \ \ \ \ \right.  ,
\]
where $Z_{1:n}\leq...\leq Z_{n:n}$ denote the order statistics pertaining to
the sample $(Z_{1},...,Z_{n})$ with the corresponding concomitants
$\delta_{\left[  1:n\right]  },...,\delta_{\left[  n:n\right]  }$ satisfying
$\delta_{\left[  j:n\right]  }=\delta_{i}$ if $Z_{j:n}=Z_{i}.$ This estimator,
known as Kaplan-Meier estimator of $F,$ may be expressed as follows%
\begin{equation}
F_{n}(x):=\sum_{i=2}^{n}W_{i,n}\mathbf{1}\left\{  Z_{i:n}\leq x\right\}  ,
\label{KM}%
\end{equation}
where, for $2\leq i\leq n,$%
\[
W_{i,n}:=\dfrac{\delta_{\left[  i:n\right]  }}{n-i+1}%
{\displaystyle\prod\limits_{j=1}^{i-1}}
\left(  \dfrac{n-j}{n-j+1}\right)  ^{\delta_{\left[  j:n\right]  }},
\]
(see, e.g., \citeauthor{RT07}, \citeyear[page 162]{RT07}). The aim of this
paper is to propose an asymptotically normal estimator for the mean of $X,$%
\[
\mu:=\mathbf{E}[X]=\int_{0}^{\infty}\overline{F}(x)dx,
\]
whose existence requires that $\gamma_{1}<1.$ The sample mean for censored
data is obtained by substituting, in the previous equation, the cdf $F$ by its
estimator $F_{n}$ to have%
\[
\widetilde{\mu}:=%
{\displaystyle\sum\limits_{i=2}^{n}}
\frac{\delta_{\left[  i:n\right]  }}{n-i+1}%
{\displaystyle\prod\limits_{j=1}^{i-1}}
\left(  \frac{n-j}{n-j+1}\right)  ^{\delta_{\left[  j:n\right]  }}Z_{i:n}.
\]
The asymptotic normality of $\widetilde{\mu}_{n}$ is established by
\cite{Stute95}, under the assumptions that the integrals
\[
I_{1}:=\int_{0}^{\infty}x^{2}\Gamma_{0}^{2}(x)dH^{\left(  1\right)  }(x)\text{
and }I_{2}:=\int_{0}^{\infty}x\left(  \int_{0}^{x}\frac{dG(y}{\overline
{H}(y)\overline{G}(y)})\right)  ^{1/2}dF(x),
\]
be finite, where $\Gamma_{0}(x):=\exp\left\{  \int_{0}^{x}dH^{\left(
0\right)  }(z)/\overline{H}(z)\right\}  $ with $H^{\left(  j\right)  }\left(
v\right)  :=\mathbb{P}\left(  Z\leq v,\text{ }\delta=j\right)  ,$ $j=0,1.$ In
the sequel, the latter functions will play a prominent role. However, when we
deal with heavy-tailed distributions, the quantities $I_{1}$ and $I_{2}$ may
be infinite. Indeed, suppose that both $F$ and $G$ are Pareto distributions,
that is $\overline{F}(x)=x^{-1/\gamma_{1}}$ and $\overline{G}(x)=x^{-1/\gamma
_{2}},$ for $x\geq1.$ This obviously gives $\overline{H}(x)=x^{-1/\gamma},$
$H^{\left(  0\right)  }(x)=\gamma(1-x^{-1/\gamma})/\gamma_{2},$ $H^{\left(
1\right)  }(x)=\gamma(1-x^{-1/\gamma})/\gamma_{1}$ and $\Gamma_{0}%
(x)=x^{1/\gamma_{2}}.$ Whenever $\left(  \gamma_{1},\gamma_{2}\right)  $ are
such that $\gamma_{1}>\gamma_{2}/\left(  1+2\gamma_{2}\right)  ,$ we readily
check that $I_{1}=I_{2}=\infty.$ In other words, the range%
\[
\mathcal{R}:=\left\{  \gamma_{1},\gamma_{2}>0:\frac{\gamma_{2}}{1+2\gamma_{2}%
}<\gamma_{1}<1\right\}  ,
\]
is not covered by the central limit theorem established by \cite{Stute95}, and
thus, another approach to handle this situation is needed. This problem was
already addressed by \cite{Peng01} for sets of complete data from heavy-tailed
distributions with tail indices lying between $1/2$ and $1.$ Note that in the
non censoring case, we have $\gamma_{1}=\gamma$ meaning that $\gamma
_{2}=\infty,$ consequently $\mathcal{R}$ reduces to Peng's range. The
consideration of the range $\mathcal{R}$ is motivated and supported from a
practical point of view as well. Indeed, as an example \cite{EFG08} analyzed
the Australian AIDS survival dataset and found that $\gamma_{1}=0.14$ and
$p=0.28$ leading to $\gamma_{2}=0.05.$ It is easily checked that these index
values belong to $\mathcal{R}$ and therefore Stute's result does not apply in
this situation. To define our new estimator, we introduce an integer sequence
$k=k_{n},$ representing a fraction of extreme order statistics, satisfying%
\begin{equation}
1<k<n,\text{ }k\rightarrow\infty\text{ and }k/n\rightarrow0\text{ as
}n\rightarrow\infty, \label{k}%
\end{equation}
and we set $h=h_{n}:=H^{-1}(1-k/n),$ where $K^{-1}(y):=\inf\left\{  x:K(x)\geq
y\right\}  ,$ $0<y<1,$ denotes the quantile function of a cdf $K.$ We start by
decomposing $\mu$ as the sum of two terms as follows:%
\[
\mu=\int_{0}^{h}\overline{F}(x)dx+\int_{h}^{\infty}\overline{F}(x)dx=:\mu
_{1}+\mu_{2},
\]
then we estimate each term separately. Integrating the first integral by parts
and changing variables in the second respectively yield%
\[
\mu_{1}=h\overline{F}(h)+\int_{0}^{h}xdF(x)\text{ and }\mu_{2}=h\overline
{F}\left(  h\right)  \int_{1}^{\infty}\frac{\overline{F}\left(  hx\right)
}{\overline{F}\left(  h\right)  }dx.
\]
By replacing $h$ and $F(x)$ by $Z_{n-k:n}$ and $F_{n}(x)$ respectively and
using formula $\left(  \ref{KM}\right)  ,$ we get%
\begin{equation}
\widehat{\mu}_{1}:=%
{\displaystyle\prod\limits_{j=1}^{n-k}}
\left(  \frac{n-j}{n-j+1}\right)  ^{\delta_{\left[  j:n\right]  }}%
Z_{n-k:n}+\sum_{i=2}^{n-k}\frac{\delta_{\left[  i:n\right]  }}{n-i+1}%
{\displaystyle\prod\limits_{j=1}^{i-1}}
\left(  \frac{n-j}{n-j+1}\right)  ^{\delta_{\left[  j:n\right]  }}Z_{i:n},
\label{mu1}%
\end{equation}
as an estimator to $\mu_{1}.$ Regarding $\mu_{2},$ we apply the well-known
Karamata theorem (see, for instance, \citeauthor{deHF06},
\citeyear[page 363]{deHF06}), to write%
\[
\mu_{2}\sim\frac{\gamma_{1}}{1-\gamma_{1}}h\overline{F}\left(  h\right)
,\text{ as }n\rightarrow\infty,\text{ }0<\gamma_{1}<1.
\]
The quantities $h$ and $\overline{F}\left(  h\right)  $ are, as above,
naturally estimated by $Z_{n-k:n}$ and%
\[
\overline{F}\left(  Z_{n-k:n}\right)  =%
{\displaystyle\prod\nolimits_{j=1}^{n-k}}
\left(  \frac{n-j}{n-j+1}\right)  ^{\delta_{\left[  j:n\right]  }},
\]
respectively. Now, it is clear that to derive an estimator to $\mu_{2},$ one
needs to estimate the tail index $\gamma_{1}.$ The general existing method,
which first appeared in \cite{BGDF07} and then developed in \cite{EFG08}, is
to consider any consistent estimator of the extremal index $\gamma$ based on
the $Z$-sample and divide it by the proportion of non-censored observations in
the tail. For instance, \cite{EFG08} adapted Hill's estimator to introduce an
estimator $\widehat{\gamma}_{1}^{(H,c)}:=\widehat{\gamma}^{H}/\widehat{p}$ to
the tail index $\gamma_{1}$ under random right censorship, where%
\[
\widehat{\gamma}^{H}:=\frac{1}{k}\sum_{i=1}^{k}\log\frac{Z_{n-i+1:n}%
}{Z_{n-k:n}}\text{ and }\widehat{p}:=\frac{1}{k}\sum_{i=1}^{k}\delta_{\left[
n-i+1:n\right]  },
\]
with $k=k_{n}$ satisfying $\left(  \ref{k}\right)  ,$ are the classical Hill
estimator and the proportion of upper non-censored observations respectively.
It is proved in \cite{BMN15} that $\widehat{p}$ consistently estimates
$p:=\gamma_{2}/\left(  \gamma_{1}+\gamma_{2}\right)  ,$ therefore
$\widehat{\gamma}_{1}^{(H,c)}$\ consistently estimates $\gamma_{1}=\gamma/p.$
The authors of \cite{BMN15} provide a Gaussian approximation leading to the
asymptotic normality of $\widehat{\gamma}_{1}^{(H,c)}$ by adopting a different
approach from that of \cite{EFG08}, who also showed that $\widehat{\gamma}%
_{1}^{(H,c)}$ is asymptotically normal. Consequently, we obtain%
\begin{equation}
\widehat{\mu}_{2}:=\frac{\widehat{\gamma}_{1}^{(H,c)}}{1-\widehat{\gamma}%
_{1}^{(H,c)}}Z_{n-k:n}%
{\displaystyle\prod\limits_{j=1}^{n-k}}
\left(  \frac{n-j}{n-j+1}\right)  ^{\delta_{\left[  j:n\right]  }},\text{ for
}\widehat{\gamma}_{1}^{(H,c)}<1, \label{mu2}%
\end{equation}
as an estimator to $\mu_{2}.$ Finally, with $\left(  \ref{mu1}\right)  $ and
$\left(  \ref{mu2}\right)  ,$ we construct our estimator $\widehat{\mu}$ of
the mean $\mu:$%
\[
\widehat{\mu}:=\sum_{i=2}^{n-k}\frac{\delta_{\left[  i:n\right]  }}{n-i+1}%
{\displaystyle\prod\limits_{j=1}^{i-1}}
\left(  \frac{n-j}{n-j+1}\right)  ^{\delta_{\left[  j:n\right]  }}Z_{i:n}+%
{\displaystyle\prod\limits_{j=1}^{n-k}}
\left(  \frac{n-j}{n-j+1}\right)  ^{\delta_{\left[  j:n\right]  }}%
\frac{Z_{n-k:n}}{1-\widehat{\gamma}_{1}^{(H,c)}}.
\]
The rest of the paper is organized as follows. In Section \ref{sec2}, we state
our main result which we prove in Section \ref{sec4}. Section \ref{sec3} is
devoted to a simulation study in which we investigate the finite sample
behavior of the newly proposed estimator $\widehat{\mu}.$ Finally, some
results, that are instrumental to our needs, are gathered in the Appendix.

\section{Main results\textbf{\label{sec2}}}

\noindent Our main result, established in the following theorem, consists in
the asymptotic normality of the newly introduced estimator $\widehat{\mu}.$ We
notice that the asymptotic normality of extreme value theory based estimators
is achieved in the second-order framework (see \citeauthor{deHS96},
\citeyear{deHS96}). Thus, it seems quite natural to suppose that cdf's $F$ and
$G$ satisfy the well-known second-order condition of regular variation. That
is, we assume that there exist constants $\tau_{j}<0$ and functions $A_{j},$
$j=1,2$ tending to zero, not changing sign near infinity and having regularly
varying absolute values with indices $\tau_{j},$ such that for any $x>0$%
\begin{equation}%
\begin{array}
[c]{l}%
\underset{t\rightarrow\infty}{\lim}\dfrac{\overline{F}(tx)/\overline
{F}(t)-x^{-1/\gamma_{1}}}{A_{1}(t)}=x^{-1/\gamma_{1}}\dfrac{x^{\tau_{1}%
/\gamma_{1}}-1}{\gamma_{1}\tau_{1}},\medskip\\
\underset{t\rightarrow\infty}{\lim}\dfrac{\overline{G}(tx)/\overline
{G}(t)-x^{-1/\gamma_{2}}}{A_{2}(t)}=x^{-1/\gamma_{2}}\dfrac{x^{\tau_{2}%
/\gamma_{2}}-1}{\gamma_{2}\tau_{2}}.
\end{array}
\label{Condi}%
\end{equation}

\begin{theorem}
\label{Theo}Assume that the second-order conditions of regular variation
$\left(  \ref{Condi}\right)  $ hold with $\gamma_{2}/\left(  1+2\gamma
_{2}\right)  <\gamma_{1}<1.$ Let $k=k_{n}$ be an integer sequence satisfying,
in addition to $\left(  \ref{k}\right)  ,$ $\lim_{n\rightarrow\infty}\sqrt
{k}A_{1}(h)<\infty$ and $\sqrt{k}h\overline{F}\left(  h\right)  \rightarrow
\infty.$ Then there exist finite constants $m$ and $\sigma^{2}>0$ such that%
\[
\frac{\sqrt{k}\left(  \widehat{\mu}-\mu\right)  }{Z_{n-k:n}\overline{F}%
_{n}(Z_{n-k:n})}\overset{d}{\rightarrow}\mathcal{N}\left(  m,\sigma
^{2}\right)  ,\text{ as }n\rightarrow\infty.
\]

\end{theorem}

\begin{remark}
We have%
\[
m:=\dfrac{\lambda_{1}}{\left(  1-p\tau_{1}\right)  \left(  1-\gamma
_{1}\right)  ^{2}}+\dfrac{\lambda_{1}}{\left(  \gamma_{1}+\tau_{1}-1\right)
\left(  1-\gamma_{1}\right)  },
\]
with $\lambda_{1}:=\lim_{n\rightarrow\infty}\sqrt{k}A_{1}(h),$ whereas the
computations of the asymptotic variance $\sigma^{2}$ are very tedious and
result in an expression that is too complicated. However, the lack of a closed
form for $\sigma^{2}$ could be overcome in applications, as both parameters
are usually estimated by the respective sample mean and variance obtained by
bootstrapping $\widehat{\mu}.$
\end{remark}

\section{Simulation study\textbf{\label{sec3}}}

\noindent We carry out a simulation study to illustrate the performance of our
estimator, through two sets of censored and censoring data, both drawn, in the
first part, from Fr\'{e}chet model%
\[
F\left(  x\right)  =\exp\left\{  -x^{-\gamma_{1}}\right\}  ,\text{ }G\left(
x\right)  =\exp\left\{  -x^{-\gamma_{2}}\right\}  ,\text{ }x\geq0,
\]
and, in the second part, from Burr model%
\[
F\left(  x\right)  =1-\left(  1+x^{1/\eta}\right)  ^{-\eta/\gamma_{1}},\text{
}G\left(  x\right)  =1-\left(  1+x^{1/\eta}\right)  ^{-\eta/\gamma_{2}},\text{
}x\geq0,
\]
where $\eta,\gamma_{1},\gamma_{2}>0.$ We fix $\eta=1/4$ and choose the values
$0.3,$ $0.4$ and $0.5$ for $\gamma_{1}.$ For the proportion of the really
observed extreme values, we take $p=0.40,$ $0.50,$ $0.60$ and $0.70.$ For each
couple $\left(  \gamma_{1},p\right)  ,$ we solve the equation $p=\gamma
_{2}/(\gamma_{1}+\gamma_{2})$ to get the pertaining $\gamma_{2}$-value. We
vary the common size $n$ of both samples $\left(  X_{1},...,X_{n}\right)  $
and $\left(  Y_{1},...,Y_{n}\right)  ,$ then for each size, we generate $1000$
independent replicates. Our overall results are taken as the empirical means
of the results obtained through the $1000$ repetitions. To determine the
optimal number (that we denote by $k^{\ast})$ of upper order statistics used
in the computation of $\widehat{\gamma}_{1}^{\left(  H,c\right)  },$ we apply
the algorithm given in page 137 of \cite{RT07}. The performance of the newly
defined estimator $\widehat{\mu}$ is evaluated in terms of absolute bias (abs
bias), mean squared error (mse) and confidence interval (conf int) accuracy
via length and coverage probability (cov prob).%

\begin{table}[tbp] \centering
\begin{tabular}
[c]{ccccccc}\hline
\multicolumn{7}{c}{$\gamma_{1}=0.3\rightarrow\mu=1.298$}\\\hline\hline
\multicolumn{7}{c}{$p=0.40$}\\\hline
$n$ & \multicolumn{1}{|c}{$\widehat{\mu}$} & abs bias & mse & conf int & cov
prob & length\\\hline
$500$ & \multicolumn{1}{|c}{$1.247$} & $0.051$ & $0.021$ & $1.043-1.450$ &
$0.88$ & $0.407$\\\hline
$1000$ & \multicolumn{1}{|c}{$1.244$} & $0.054$ & $0.020$ & $1.099-1.389$ &
$0.88$ & $0.291$\\\hline
$1500$ & \multicolumn{1}{|c}{$1.233$} & $0.065$ & $0.005$ & $1.119-1.346$ &
$0.80$ & $0.227$\\\hline
$2000$ & \multicolumn{1}{|c}{$1.231$} & $0.067$ & $0.005$ & $1.135-1.328$ &
$0.74$ & $0.193$\\\hline
\multicolumn{7}{c}{$p=0.50$}\\\hline
$500$ & \multicolumn{1}{|c}{$1.248$} & $0.050$ & $0.008$ & $1.049-1.447$ &
$0.96$ & $0.399$\\\hline
$1000$ & \multicolumn{1}{|c}{$1.247$} & $0.051$ & $0.004$ & $1.107-1.387$ &
$0.90$ & $0.280$\\\hline
$1500$ & \multicolumn{1}{|c}{$1.250$} & $0.048$ & $0.003$ & $1.134-1.365$ &
$0.90$ & $0.231$\\\hline
$2000$ & \multicolumn{1}{|c}{$1.248$} & $0.050$ & $0.003$ & $1.146-1.350$ &
$0.86$ & $0.204$\\\hline
\multicolumn{7}{c}{$p=0.60$}\\\hline
$500$ & \multicolumn{1}{|c}{$1.254$} & $0.044$ & $0.009$ & $1.050-1.458$ &
$0.90$ & $0.408$\\\hline
$1000$ & \multicolumn{1}{|c}{$1.257$} & $0.041$ & $0.003$ & $1.119-1.395$ &
$0.94$ & $0.275$\\\hline
$1500$ & \multicolumn{1}{|c}{$1.266$} & $0.032$ & $0.002$ & $1.153-1.379$ &
$0.96$ & $0.226$\\\hline
$2000$ & \multicolumn{1}{|c}{$1.264$} & $0.034$ & $0.002$ & $1.164-1.364$ &
$0.92$ & $0.200$\\\hline
\multicolumn{7}{c}{$p=0.70$}\\\hline
$500$ & \multicolumn{1}{|c}{$1.265$} & $0.033$ & $0.003$ & $1.069-1.460$ &
$0.97$ & $0.391$\\\hline
$1000$ & \multicolumn{1}{|c}{$1.269$} & $0.029$ & $0.002$ & $1.123-1.415$ &
$0.96$ & $0.291$\\\hline
$1500$ & \multicolumn{1}{|c}{$1.279$} & $0.019$ & $0.001$ & $1.162-1.395$ &
$0.98$ & $0.233$\\\hline
$2000$ & \multicolumn{1}{|c}{$1.278$} & $0.020$ & $0.001$ & $1.178-1.377$ &
$0.96$ & $0.199$\\\hline
&  &  &  &  &  &
\end{tabular}
\caption{Absolute bias, mean squared error and 95\%-confidence interval accuracy of the mean estimator based on 1000 right-censored samples from Fréchet model with shape parameter 0.3}\label{Tab1}%
\end{table}%
%

\begin{table}[tbp] \centering
\begin{tabular}
[c]{ccccccc}\hline
\multicolumn{7}{c}{$\gamma_{1}=0.4\rightarrow\mu=$1.489}\\\hline\hline
\multicolumn{7}{c}{$p=0.40$}\\\hline
$n$ & \multicolumn{1}{|c}{$\widehat{\mu}$} & abs bias & mse & conf int & cov
prob & length\\\hline
\multicolumn{1}{r}{$500$} & \multicolumn{1}{|c}{$1.370$} & $0.120$ & $0.074$ &
$1.147-1.593$ & $0.71$ & $0.446$\\\hline
\multicolumn{1}{r}{$1000$} & \multicolumn{1}{|c}{$1.377$} & $0.112$ & $0.048$
& $1.217-1.536$ & $0.57$ & $0.319$\\\hline
\multicolumn{1}{r}{$1500$} & \multicolumn{1}{|c}{$1.367$} & $0.122$ & $0.019$
& $1.241-1.493$ & $0.48$ & $0.252$\\\hline
\multicolumn{1}{r}{$2000$} & \multicolumn{1}{|c}{$1.363$} & $0.126$ & $0.018$
& $1.256-1.470$ & $0.36$ & $0.214$\\\hline
\multicolumn{7}{c}{$p=0.50$}\\\hline
\multicolumn{1}{r}{$500$} & \multicolumn{1}{|c}{$1.396$} & $0.093$ & $0.027$ &
$1.169-1.624$ & $0.81$ & $0.455$\\\hline
\multicolumn{1}{r}{$1000$} & \multicolumn{1}{|c}{$1.394$} & $0.095$ & $0.018$
& $1.237-1.551$ & $0.66$ & $0.313$\\\hline
\multicolumn{1}{r}{$1500$} & \multicolumn{1}{|c}{$1.392$} & $0.097$ & $0.012$
& $1.264-1.521$ & $0.65$ & $0.257$\\\hline
\multicolumn{1}{r}{$2000$} & \multicolumn{1}{|c}{$1.389$} & $0.101$ & $0.012$
& $1.275-1.502$ & $0.55$ & $0.227$\\\hline
\multicolumn{7}{c}{$p=0.60$}\\\hline
\multicolumn{1}{r}{$500$} & \multicolumn{1}{|c}{$1.407$} & $0.082$ & $0.013$ &
$1.189-1.625$ & $0.89$ & $0.436$\\\hline
\multicolumn{1}{r}{$1000$} & \multicolumn{1}{|c}{$1.405$} & $0.084$ & $0.010$
& $1.251-1.559$ & $0.77$ & $0.308$\\\hline
\multicolumn{1}{r}{$1500$} & \multicolumn{1}{|c}{$1.419$} & $0.070$ & $0.007$
& $1.292-1.546$ & $0.84$ & $0.254$\\\hline
\multicolumn{1}{r}{$2000$} & \multicolumn{1}{|c}{$1.418$} & $0.071$ & $0.007$
& $1.308-1.529$ & $0.71$ & $0.222$\\\hline
\multicolumn{7}{c}{$p=0.70$}\\\hline
\multicolumn{1}{r}{$500$} & \multicolumn{1}{|c}{$1.420$} & $0.069$ & $0.010$ &
$1.199-1.641$ & $0.92$ & $0.442$\\\hline
\multicolumn{1}{r}{$1000$} & \multicolumn{1}{|c}{$1.433$} & $0.056$ & $0.006$
& $1.273-1.593$ & $0.86$ & $0.320$\\\hline
\multicolumn{1}{r}{$1500$} & \multicolumn{1}{|c}{$1.443$} & $0.046$ & $0.004$
& $1.312-1.575$ & $0.90$ & $0.263$\\\hline
\multicolumn{1}{r}{$2000$} & \multicolumn{1}{|c}{$1.442$} & $0.047$ & $0.004$
& $1.329-1.554$ & $0.89$ & $0.226$\\\hline
&  &  &  &  &  &
\end{tabular}
\caption{Absolute bias, mean squared error and 95\%-confidence interval accuracy of the mean estimator based on 1000 right-censored samples from Fréchet model with shape parameter 0.4}.\label{Tab2}%
\end{table}%
%

\begin{table}[tbp] \centering
\begin{tabular}
[c]{ccccccc}\hline
\multicolumn{7}{c}{$\gamma_{1}=0.5\rightarrow\mu=1.772$}\\\hline\hline
\multicolumn{7}{c}{$p=0.40$}\\\hline
$n$ & \multicolumn{1}{|c}{$\widehat{\mu}$} & abs bias & mse & conf int & cov
prob & length\\\hline
\multicolumn{1}{r}{$500$} & \multicolumn{1}{|c}{$1.566$} & $0.206$ & $0.398$ &
$1.262-1.870$ & $0.52$ & $0.608$\\\hline
\multicolumn{1}{r}{$1000$} & \multicolumn{1}{|c}{$1.550$} & $0.223$ & $0.176$
& $1.372-1.727$ & $0.28$ & $0.355$\\\hline
\multicolumn{1}{r}{$1500$} & \multicolumn{1}{|c}{$1.559$} & $0.214$ & $0.064$
& $1.415-1.703$ & $0.20$ & $0.289$\\\hline
\multicolumn{1}{r}{$2000$} & \multicolumn{1}{|c}{$1.549$} & $0.224$ & $0.061$
& $1.426-1.671$ & $0.13$ & $0.245$\\\hline
\multicolumn{7}{c}{$p=0.50$}\\\hline
\multicolumn{1}{r}{$500$} & \multicolumn{1}{|c}{$1.577$} & $0.195$ & $0.180$ &
$1.309-1.846$ & $0.53$ & $0.537$\\\hline
\multicolumn{1}{r}{$1000$} & \multicolumn{1}{|c}{$1.573$} & $0.199$ & $0.139$
& $1.386-1.761$ & $0.37$ & $0.375$\\\hline
\multicolumn{1}{r}{$1500$} & \multicolumn{1}{|c}{$1.578$} & $0.195$ & $0.051$
& $1.430-1.725$ & $0.20$ & $0.294$\\\hline
\multicolumn{1}{r}{$2000$} & \multicolumn{1}{|c}{$1.576$} & $0.196$ & $0.044$
& $1.447-1.706$ & $0.22$ & $0.259$\\\hline
\multicolumn{7}{c}{$p=0.60$}\\\hline
\multicolumn{1}{r}{$500$} & \multicolumn{1}{|c}{$1.626$} & $0.147$ & $0.128$ &
$1.362-1.889$ & $0.65$ & $0.527$\\\hline
\multicolumn{1}{r}{$1000$} & \multicolumn{1}{|c}{$1.617$} & $0.155$ & $0.034$
& $1.430-1.805$ & $0.56$ & $0.375$\\\hline
\multicolumn{1}{r}{$1500$} & \multicolumn{1}{|c}{$1.606$} & $0.166$ & $0.033$
& $1.465-1.747$ & $0.34$ & $0.282$\\\hline
\multicolumn{1}{r}{$2000$} & \multicolumn{1}{|c}{$1.622$} & $0.150$ & $0.029$
& $1.494-1.751$ & $0.34$ & $0.258$\\\hline
\multicolumn{7}{c}{$p=0.70$}\\\hline
\multicolumn{1}{r}{$500$} & \multicolumn{1}{|c}{$1.632$} & $0.141$ & $0.046$ &
$1.375-1.888$ & $0.72$ & $0.513$\\\hline
\multicolumn{1}{r}{$1000$} & \multicolumn{1}{|c}{$1.646$} & $0.126$ & $0.024$
& $1.459-1.833$ & $0.70$ & $0.370$\\\hline
\multicolumn{1}{r}{$1500$} & \multicolumn{1}{|c}{$1.668$} & $0.104$ & $0.017$
& $1.516-1.821$ & $0.68$ & $0.305$\\\hline
\multicolumn{1}{r}{$2000$} & \multicolumn{1}{|c}{$1.666$} & $0.107$ & $0.016$
& $1.535-1.797$ & $0.57$ & $0.262$\\\hline
&  &  &  &  &  &
\end{tabular}
\caption{Absolute bias, mean squared error and 95\%-confidence interval accuracy of the mean estimator based on 1000 right-censored samples from Fréchet model with shape parameter 0.5}.\label{Tab3}%
\end{table}%
%

\begin{table}[tbp] \centering
\begin{tabular}
[c]{ccccccc}\hline
\multicolumn{7}{c}{$\gamma_{1}=0.3\rightarrow\mu=1.228$}\\\hline\hline
\multicolumn{7}{c}{$p=0.40$}\\\hline
$n$ & \multicolumn{1}{|c}{$\widehat{\mu}$} & abs bias & mse & conf int & cov
prob & length\\\hline
$500$ & \multicolumn{1}{|c}{$1.186$} & $0.042$ & $0.077$ & $0.972-1.399$ &
$0.90$ & $0.428$\\\hline
$1000$ & \multicolumn{1}{|c}{$1.179$} & $0.049$ & $0.019$ & $1.038-1.32$ &
$0.80$ & $0.282$\\\hline
$1500$ & \multicolumn{1}{|c}{$1.163$} & $0.064$ & $0.005$ & $1.053-1.273$ &
$0.80$ & $0.220$\\\hline
$2000$ & \multicolumn{1}{|c}{$1.164$} & $0.063$ & $0.005$ & $1.068-1.261$ &
$0.72$ & $0.193$\\\hline
\multicolumn{7}{c}{$p=0.50$}\\\hline
\multicolumn{1}{r}{$500$} & \multicolumn{1}{|c}{$1.186$} & $0.042$ & $0.009$ &
$0.991-1.380$ & $0.94$ & $0.388$\\\hline
\multicolumn{1}{r}{$1000$} & \multicolumn{1}{|c}{$1.173$} & $0.054$ & $0.004$
& $1.039-1.308$ & $0.93$ & $0.269$\\\hline
\multicolumn{1}{r}{$1500$} & \multicolumn{1}{|c}{$1.068$} & $0.047$ & $0.003$
& $1.180-1.292$ & $0.88$ & $0.224$\\\hline
\multicolumn{1}{r}{$2000$} & \multicolumn{1}{|c}{$1.181$} & $0.046$ & $0.003$
& $1.086-1.276$ & $0.86$ & $0.190$\\\hline
\multicolumn{7}{c}{$p=0.60$}\\\hline
\multicolumn{1}{r}{$500$} & \multicolumn{1}{|c}{$1.184$} & $0.043$ & $0.004$ &
$0.997-1.371$ & $0.95$ & $0.374$\\\hline
\multicolumn{1}{r}{$1000$} & \multicolumn{1}{|c}{$1.192$} & $0.036$ & $0.002$
& $1.058-1.326$ & $0.96$ & $0.268$\\\hline
\multicolumn{1}{r}{$1500$} & \multicolumn{1}{|c}{$1.196$} & $0.031$ & $0.002$
& $1.088-1.305$ & $0.96$ & $0.217$\\\hline
\multicolumn{1}{r}{$2000$} & \multicolumn{1}{|c}{$1.194$} & $0.034$ & $0.002$
& $1.099-1.288$ & $0.92$ & $0.190$\\\hline
\multicolumn{7}{c}{$p=0.70$}\\\hline
\multicolumn{1}{r}{$500$} & \multicolumn{1}{|c}{$1.198$} & $0.029$ & $0.003$ &
$1.012-1.384$ & $0.97$ & $0.373$\\\hline
\multicolumn{1}{r}{$1000$} & \multicolumn{1}{|c}{$1.200$} & $0.028$ & $0.001$
& $1.066-1.334$ & $0.98$ & $0.269$\\\hline
\multicolumn{1}{r}{$1500$} & \multicolumn{1}{|c}{$1.208$} & $0.020$ & $0.001$
& $1.098-1.317$ & $0.98$ & $0.219$\\\hline
\multicolumn{1}{r}{$2000$} & \multicolumn{1}{|c}{$1.207$} & $0.021$ & $0.001$
& $1.113-1.301$ & $0.98$ & $0.188$\\\hline
&  &  &  &  &  &
\end{tabular}
\caption{Absolute bias, mean squared error and 95\%-confidence interval accuracy of the mean estimator based on 1000 right-censored samples from Burr model with shape parameter 0.3}\label{Tab4}%
\end{table}%
%

\begin{table}[tbp] \centering
\begin{tabular}
[c]{ccccccc}\hline
\multicolumn{7}{c}{$\gamma_{1}=0.4\rightarrow\mu=1.498$}\\\hline\hline
\multicolumn{7}{c}{$p=0.40$}\\\hline
$n$ & \multicolumn{1}{|c}{$\widehat{\mu}$} & abs bias & mse & conf int & cov
prob & length\\\hline
\multicolumn{1}{r}{$500$} & \multicolumn{1}{|c}{$1.426$} & $0.071$ & $0.093$ &
$1.193-1.660$ & $0.76$ & $0.466$\\\hline
\multicolumn{1}{r}{$1000$} & \multicolumn{1}{|c}{$1.388$} & $0.110$ & $0.033$
& $1.224-1.551$ & $0.58$ & $0.327$\\\hline
\multicolumn{1}{r}{$1500$} & \multicolumn{1}{|c}{$1.374$} & $0.124$ & $0.020$
& $1.248-1.499$ & $0.44$ & $0.252$\\\hline
\multicolumn{1}{r}{$2000$} & \multicolumn{1}{|c}{$1.374$} & $0.123$ & $0.019$
& $1.268-1.480$ & $0.29$ & $0.212$\\\hline
\multicolumn{7}{c}{$p=0.50$}\\\hline
\multicolumn{1}{r}{$500$} & \multicolumn{1}{|c}{$1.402$} & $0.096$ & $0.047$ &
$1.176-1.627$ & $0.80$ & $0.451$\\\hline
\multicolumn{1}{r}{$1000$} & \multicolumn{1}{|c}{$1.389$} & $0.109$ & $0.017$
& $1.231-1.546$ & $0.64$ & $0.316$\\\hline
\multicolumn{1}{r}{$1500$} & \multicolumn{1}{|c}{$1.401$} & $0.097$ & $0.012$
& $1.272-1.530$ & $0.66$ & $0.258$\\\hline
\multicolumn{1}{r}{$2000$} & \multicolumn{1}{|c}{$1.402$} & $0.096$ & $0.011$
& $1.292-1.511$ & $0.53$ & $0.219$\\\hline
\multicolumn{7}{c}{$p=0.60$}\\\hline
\multicolumn{1}{r}{$500$} & \multicolumn{1}{|c}{$1.422$} & $0.076$ & $0.043$ &
$1.186-1.657$ & $0.85$ & $0.471$\\\hline
\multicolumn{1}{r}{$1000$} & \multicolumn{1}{|c}{$1.421$} & $0.077$ & $0.009$
& $1.261-1.581$ & $0.86$ & $0.320$\\\hline
\multicolumn{1}{r}{$1500$} & \multicolumn{1}{|c}{$1.429$} & $0.069$ & $0.007$
& $1.302-1.556$ & $0.80$ & $0.254$\\\hline
\multicolumn{1}{r}{$2000$} & \multicolumn{1}{|c}{$1.427$} & $0.071$ & $0.006$
& $1.316-1.538$ & $0.76$ & $0.223$\\\hline
\multicolumn{7}{c}{$p=0.70$}\\\hline
\multicolumn{1}{r}{$500$} & \multicolumn{1}{|c}{$1.436$} & $0.061$ & $0.009$ &
$1.214-1.658$ & $0.94$ & $0.444$\\\hline
\multicolumn{1}{r}{$1000$} & \multicolumn{1}{|c}{$1.441$} & $0.057$ & $0.006$
& $1.285-1.597$ & $0.92$ & $0.312$\\\hline
\multicolumn{1}{r}{$1500$} & \multicolumn{1}{|c}{$1.451$} & $0.047$ & $0.004$
& $1.322-1.580$ & $0.91$ & $0.259$\\\hline
\multicolumn{1}{r}{$2000$} & \multicolumn{1}{|c}{$1.449$} & $0.049$ & $0.004$
& $1.340-1.558$ & $0.88$ & $0.218$\\\hline
&  &  &  &  &  &
\end{tabular}
\caption{Absolute bias, mean squared error and 95\%-confidence interval accuracy of the mean estimator based on 1000 right-censored samples from Burr model with shape parameter 0.4}.\label{Tab5}%
\end{table}%
%

\begin{table}[tbp] \centering
\begin{tabular}
[c]{ccccccc}\hline
\multicolumn{7}{c}{$\gamma_{1}=0.5\rightarrow\mu=1.854$}\\\hline\hline
\multicolumn{7}{c}{$p=0.40$}\\\hline
$n$ & \multicolumn{1}{|c}{$\widehat{\mu}$} & abs bias & mse & conf int & cov
prob & length\\\hline
$500$ & \multicolumn{1}{|c}{$1.654$} & $0.200$ & $0.760$ & $1.330-1.978$ &
$0.50$ & $0.649$\\\hline
$1000$ & \multicolumn{1}{|c}{$1.648$} & $0.206$ & $0.114$ & $1.460-1.836$ &
$0.26$ & $0.375$\\\hline
$1500$ & \multicolumn{1}{|c}{$1.630$} & $0.224$ & $0.098$ & $1.478-1.782$ &
$0.14$ & $0.304$\\\hline
$2000$ & \multicolumn{1}{|c}{$1.621$} & $0.233$ & $0.090$ & $1.491-1.752$ &
$0.14$ & $0.260$\\\hline
\multicolumn{7}{c}{$p=0.50$}\\\hline
$500$ & \multicolumn{1}{|c}{$1.603$} & $0.252$ & $0.554$ & $1.253-1.952$ &
$0.67$ & $0.700$\\\hline
$1000$ & \multicolumn{1}{|c}{$1.658$} & $0.196$ & $0.090$ & $1.470-1.847$ &
$0.34$ & $0.378$\\\hline
$1500$ & \multicolumn{1}{|c}{$1.653$} & $0.202$ & $0.049$ & $1.501-1.804$ &
$0.25$ & $0.303$\\\hline
$2000$ & \multicolumn{1}{|c}{$1.656$} & $0.198$ & $0.045$ & $1.530-1.782$ &
$0.22$ & $0.252$\\\hline
\multicolumn{7}{c}{$p=0.60$}\\\hline
$500$ & \multicolumn{1}{|c}{$1.688$} & $0.166$ & $0.066$ & $1.417-1.959$ &
$0.67$ & $0.542$\\\hline
$1000$ & \multicolumn{1}{|c}{$1.693$} & $0.161$ & $0.036$ & $1.508-1.879$ &
$0.54$ & $0.371$\\\hline
$1500$ & \multicolumn{1}{|c}{$1.695$} & $0.159$ & $0.031$ & $1.544-1.846$ &
$0.39$ & $0.301$\\\hline
$2000$ & \multicolumn{1}{|c}{$1.705$} & $0.149$ & $0.027$ & $1.576-1.834$ &
$0.34$ & $0.258$\\\hline
\multicolumn{7}{c}{$p=0.70$}\\\hline
$500$ & \multicolumn{1}{|c}{$1.737$} & $0.117$ & $0.060$ & $1.462-2.012$ &
$0.77$ & $0.550$\\\hline
$1000$ & \multicolumn{1}{|c}{$1.737$} & $0.117$ & $0.036$ & $1.547-1.927$ &
$0.74$ & $0.380$\\\hline
$1500$ & \multicolumn{1}{|c}{$1.749$} & $0.105$ & $0.016$ & $1.593-1.904$ &
$0.70$ & $0.311$\\\hline
$2000$ & \multicolumn{1}{|c}{$1.753$} & $0.101$ & $0.014$ & $1.621-1.885$ &
$0.60$ & $0.264$\\\hline
&  &  &  &  &  &
\end{tabular}
\caption{Absolute bias, mean squared error and 95\%-confidence interval accuracy of the mean estimator based on 1000 right-censored samples from Burr model with shape parameter 0.5}.\label{Tab6}%
\end{table}%

\noindent The results, summarized in Tables \ref{Tab1}, \ref{Tab2} and
\ref{Tab3} for Fr\'{e}chet model and Table \ref{Tab4}, \ref{Tab5} and
\ref{Tab6} for Burr distribution, show that the same conclusions might be
drawn in both cases. As expected, the sample size influences the estimation in
the sense that the larger $n$ gets, the better the estimation is. On the other
hand, it is clear that the estimation accuracy increases when the censoring
percentage decreases, which seems logical. Moreover, the estimator performs
best for the smaller value of the tail index, as we can see from Tables
\ref{Tab1} and \ref{Tab4}. Finally, many simulations realized with extreme
value indices larger than $0.5,$ but whose results are not reported here, show
that the estimator behaves poorly especially when the censorship proportion is high.

\section{Proofs\textbf{\label{sec4}}}

\noindent We begin by a brief introduction on some uniform empirical processes
under random censoring. The empirical counterparts of $\ H^{\left(  j\right)
}$ $(j=0,1)$ are defined, for $v\geq0,$ by%
\[
H_{n}^{\left(  j\right)  }(v):=\frac{1}{n}\sum_{i=1}^{n}\mathbf{1}\left\{
Z_{i}\leq v,\delta_{i}=j\right\}  ,\text{ }j=0,1.
\]
In the sequel, we will use the following two empirical processes%
\[
\sqrt{n}\left(  \overline{H}_{n}^{\left(  j\right)  }(v)-\overline{H}^{\left(
j\right)  }(v)\right)  ,\text{ }j=0,1;\text{ }v\geq0,
\]
which may be represented, almost surely, by a uniform empirical process.
Indeed, let us define, for each $i=1,...,n,$ the following rv%
\[
U_{i}:=\delta_{i}H^{\left(  1\right)  }(Z_{i})+(1-\delta_{i})(\theta
+H^{\left(  0\right)  }(Z_{i})).
\]
From \cite{EK92}, the rv's $U_{1},...,U_{n}$ are iid $(0,1)$-uniform. The
empirical cdf and the uniform empirical process based upon $U_{1},...,U_{n}$
are respectively denoted by%
\[
\mathbb{U}_{n}(s):\mathbb{=}\frac{1}{n}\sum_{i=1}^{n}\mathbf{1}\left\{
U_{i}\leq s\right\}  \text{ and }\alpha_{n}(s):=\sqrt{n}(\mathbb{U}%
_{n}(s)-s),\text{ }0\leq s\leq1.
\]
\cite{DE96} state that almost surely%
\[
H_{n}^{\left(  0\right)  }(v)=\mathbb{U}_{n}(H^{\left(  0\right)  }%
(v)+\theta)-\mathbb{U}_{n}(\theta),\text{ for }0<H^{\left(  0\right)
}(v)<1-\theta,
\]
and%
\[
H_{n}^{\left(  1\right)  }(v)=\mathbb{U}_{n}(H^{\left(  1\right)  }(v)),\text{
for }0<H^{\left(  1\right)  }(v)<\theta.
\]
It is easy to verify that almost surely%
\begin{equation}
\beta_{n}\left(  v\right)  :=\sqrt{n}\left(  \overline{H}_{n}^{\left(
1\right)  }(v)-\overline{H}^{\left(  1\right)  }(v)\right)  =\alpha_{n}\left(
\theta\right)  -\alpha_{n}\left(  \theta-\overline{H}^{\left(  1\right)
}(v)\right)  ,\text{ for }0<\overline{H}^{\left(  1\right)  }(v)<\theta,
\label{rep-H1}%
\end{equation}
and%
\begin{equation}
\widetilde{\beta}_{n}\left(  v\right)  :=\sqrt{n}\left(  \overline{H}%
_{n}^{\left(  0\right)  }(v)-\overline{H}^{\left(  0\right)  }(v)\right)
=-\alpha_{n}\left(  1-\overline{H}^{\left(  0\right)  }(v)\right)  ,\text{ for
}0<\overline{H}^{\left(  0\right)  }(v)<1-\theta. \label{rep-H0}%
\end{equation}
Our methodology strongly relies on the well-known Gaussian approximation given
in Corollary 2.1 by \cite{CHM86}. It says that: on the probability space
$(\Omega,\mathcal{A},\mathbb{P)},$ there exists a sequence of Brownian bridges
$\left\{  B_{n}(s);\text{ }0\leq s\leq1\right\}  $ such that for every
$0\leq\zeta<1/4,$%
\begin{equation}
\underset{1/n\leq s\leq1}{\sup}\frac{n^{\zeta}\left\vert \alpha_{n}%
(1-s)-B_{n}(1-s)\right\vert }{s^{1/2-\zeta}}=O_{\mathbb{P}}(1). \label{Cs}%
\end{equation}
For the increments $\alpha_{n}(\theta)-\alpha_{n}(\theta-s),$ we will need an
approximation of the same type as $\left(  \ref{Cs}\right)  $. Following
similar arguments, mutatis mutandis, as those used in the proofs of assertions
(2.2) of Theorem 2.1 and (2.8) of Theorem 2.2 in \cite{CHM86}, we may show
that, for every $0<\theta<1$ and $0\leq\zeta<1/4,$ we have%
\begin{equation}
\underset{1/n\leq s\leq\theta}{\sup}\frac{n^{\zeta}\left\vert \left\{
\alpha_{n}(\theta)-\alpha_{n}(\theta-s)\right\}  -\left\{  B_{n}\left(
\theta\right)  -B_{n}(\theta-s)\right\}  \right\vert }{s^{1/2-\zeta}%
}=O_{\mathbb{P}}(1). \label{Necir}%
\end{equation}

\subsection{Proof of Theorem \ref{Theo}}

\noindent Observe that $\widehat{\mu}-\mu=\left(  \widehat{\mu}_{1}-\mu
_{1}\right)  +\left(  \widehat{\mu}_{2}-\mu_{2}\right)  ,$ where%
\[
\widehat{\mu}_{1}-\mu_{1}=\int_{0}^{Z_{n-k:n}}\overline{F}_{n}(x)dx-\int
_{0}^{h}\overline{F}(x)dx,
\]
and%
\[
\widehat{\mu}_{2}-\mu_{2}=%
{\displaystyle\prod\limits_{j=1}^{n-k}}
\left(  1-\frac{\delta_{\left[  j:n\right]  }}{n-j+1}\right)  \frac
{\widehat{\gamma}_{1}^{(H,c)}}{1-\widehat{\gamma}_{1}^{(H,c)}}Z_{n-k:n}%
-\int_{h}^{\infty}\overline{F}(x)dx.
\]
It is clear that%
\[
\widehat{\mu}_{1}-\mu_{1}=\int_{0}^{Z_{n-k:n}}(\overline{F}_{n}(x)-\overline
{F}(x))dx-\int_{Z_{n-k:n}}^{h}\overline{F}(x)dx.
\]
In view of Proposition 5 combined with equation $\left(  4.9\right)  $ in
\cite{C96}, we have for any $x\leq Z_{n-k:n},$%
\begin{align*}
&  \frac{\overline{F}_{n}(x)-\overline{F}\left(  x\right)  }{\overline
{F}\left(  x\right)  }=\\
&  \int_{0}^{x}\frac{d\left(  \overline{H}_{n}^{\left(  1\right)  }\left(
v\right)  -\overline{H}^{\left(  1\right)  }\left(  v\right)  \right)
}{\overline{H}\left(  v\right)  }-\int_{0}^{x}\frac{\overline{H}_{n}\left(
v\right)  -\overline{H}\left(  v\right)  }{\overline{H}^{2}\left(  v\right)
}d\overline{H}^{\left(  1\right)  }\left(  v\right)  +O_{\mathbb{P}}\left(
1/k\right)  .
\end{align*}
Integrating the first integral by parts yields%
\begin{align*}
&  \frac{\overline{F}_{n}(x)-\overline{F}\left(  x\right)  }{\overline
{F}\left(  x\right)  }=\frac{\overline{H}_{n}^{\left(  1\right)  }\left(
x\right)  -\overline{H}^{\left(  1\right)  }\left(  x\right)  }{\overline
{H}\left(  x\right)  }-\left(  \overline{H}_{n}^{\left(  1\right)  }\left(
0\right)  -\overline{H}^{\left(  1\right)  }\left(  0\right)  \right) \\
&  +\int_{0}^{x}\frac{\overline{H}_{n}^{\left(  1\right)  }\left(  v\right)
-\overline{H}^{\left(  1\right)  }\left(  v\right)  }{\overline{H}^{2}\left(
v\right)  }d\overline{H}\left(  v\right)  -\int_{0}^{x}\frac{\overline{H}%
_{n}\left(  v\right)  -\overline{H}\left(  v\right)  }{\overline{H}^{2}\left(
v\right)  }d\overline{H}^{\left(  1\right)  }\left(  v\right)  +O_{\mathbb{P}%
}\left(  1/k\right)  .
\end{align*}
Recall that%
\[
\sqrt{n}\left(  \overline{H}_{n}\left(  v\right)  -\overline{H}\left(
v\right)  \right)  =\sqrt{n}\left(  \overline{H}_{n}^{\left(  1\right)
}\left(  v\right)  -\overline{H}^{\left(  1\right)  }\left(  v\right)
\right)  +\sqrt{n}\left(  \overline{H}_{n}^{\left(  0\right)  }\left(
v\right)  -\overline{H}^{\left(  0\right)  }\left(  v\right)  \right)  ,
\]
which by representations $\left(  \ref{rep-H1}\right)  $ and $\left(
\ref{rep-H0}\right)  $ becomes%
\[
\sqrt{n}\left(  \overline{H}_{n}\left(  v\right)  -\overline{H}\left(
v\right)  \right)  =\alpha_{n}\left(  \theta\right)  -\alpha_{n}\left(
\theta-\overline{H}^{\left(  1\right)  }\left(  v\right)  \right)  -\alpha
_{n}\left(  1-\overline{H}^{\left(  0\right)  }\left(  v\right)  \right)  .
\]
Furthermore, from the classical central limit theorem, we have $\overline
{H}_{n}^{\left(  1\right)  }\left(  0\right)  -\overline{H}^{\left(  1\right)
}\left(  0\right)  =O_{\mathbb{P}}\left(  n^{-1/2}\right)  .$ Therefore, we
have%
\begin{align}
&  \frac{\overline{F}_{n}(x)-\overline{F}\left(  x\right)  }{\overline
{F}\left(  x\right)  }=\frac{1}{\sqrt{n}}\frac{\beta_{n}\left(  x\right)
}{\overline{H}\left(  x\right)  }+\frac{1}{\sqrt{n}}\int_{0}^{x}\frac
{\beta_{n}\left(  v\right)  }{\overline{H}^{2}\left(  v\right)  }d\overline
{H}\left(  v\right) \label{ratio}\\
&  -\frac{1}{\sqrt{n}}\int_{0}^{x}\frac{\beta_{n}\left(  v\right)
+\widetilde{\beta}_{n}\left(  v\right)  }{\overline{H}^{2}\left(  v\right)
}d\overline{H}^{\left(  1\right)  }\left(  v\right)  +O_{\mathbb{P}}\left(
1/k\right)  +O_{\mathbb{P}}\left(  1/\sqrt{n}\right)  .\nonumber
\end{align}
By letting $a_{n}:=\left(  k/n\right)  ^{1/2}/\left(  h\overline{F}(h)\right)
,$ it is easy to verify that%
\[
\frac{\sqrt{k}\left(  \widehat{\mu}_{1}-\mu_{1}\right)  }{h\overline{F}(h)}=%
{\displaystyle\sum_{i=1}^{6}}
T_{ni},
\]
where%
\[%
\begin{array}
[c]{cl}%
T_{n1}:= & a_{n}%
{\displaystyle\int_{0}^{Z_{n-k:n}}}
\dfrac{\beta_{n}\left(  x\right)  }{\overline{H}(x)}\overline{F}%
(x)dx,\bigskip\\
T_{n2}:= & a_{n}%
{\displaystyle\int_{0}^{Z_{n-k:n}}}
\left\{
{\displaystyle\int_{0}^{x}}
\dfrac{\beta_{n}\left(  v\right)  }{\overline{H}^{2}(v)}d\overline
{H}(v)\right\}  \overline{F}(x)dx,\bigskip\\
T_{n3}:= & -a_{n}%
{\displaystyle\int_{0}^{Z_{n-k:n}}}
\left\{
{\displaystyle\int_{0}^{x}}
\dfrac{\beta_{n}\left(  v\right)  +\widetilde{\beta}_{n}\left(  v\right)
}{\overline{H}^{2}(v)}d\overline{H}^{\left(  1\right)  }(v)\right\}
\overline{F}(x)dx,\bigskip\\
T_{n4}:= & a_{n}O_{\mathbb{P}}\left(  \sqrt{n}/k\right)
{\displaystyle\int_{0}^{Z_{n-k:n}}}
\overline{F}(x)dx,\bigskip\\
T_{n5}:= & -a_{n}\sqrt{n}%
{\displaystyle\int_{Z_{n-k:n}}^{h}}
\overline{F}(x)dx\ \ \text{and \ }T_{n6}:=O_{\mathbb{P}}\left(  a_{n}\right)
.
\end{array}
\]
By using the Gaussian approximation $\left(  \ref{Necir}\right)  ,$ we obtain%
\begin{align*}
T_{n1}  &  =a_{n}%
{\displaystyle\int_{0}^{Z_{n-k:n}}}
\dfrac{\overline{F}(x)}{\overline{H}(x)}\mathbf{B}_{n}\left(  x\right)  dx\\
&  +o_{\mathbb{P}}\left(  1\right)  a_{n}%
{\displaystyle\int_{0}^{Z_{n-k:n}}}
\dfrac{\left(  \overline{H}^{\left(  1\right)  }(x)\right)  ^{1/2}}%
{\overline{H}(x)}\overline{F}(x)dx,
\end{align*}
where%
\begin{equation}
\mathbf{B}_{n}\left(  x\right)  :=B_{n}(\theta)-B_{n}\left(  \theta
-\overline{H}^{\left(  1\right)  }\left(  x\right)  \right)  ,\text{ for
}0<\overline{H}^{\left(  1\right)  }\left(  x\right)  <\theta. \label{Bn}%
\end{equation}
Next, we show that the second term of $T_{n1}$ tends to zero in probability,
leading to%
\[
T_{n1}=a_{n}%
{\displaystyle\int_{0}^{Z_{n-k:n}}}
\dfrac{\overline{F}(v)}{\overline{H}(v)}\mathbf{B}_{n}\left(  x\right)
dx+o_{\mathbb{P}}\left(  1\right)  .
\]
Let $0\leq\zeta<1/4$ and note that since $\overline{H}=\overline{H}^{\left(
0\right)  }+\overline{H}^{\left(  1\right)  },$ then $\overline{H}^{\left(
1\right)  }\leq\overline{H}$ and%
\[
O_{\mathbb{P}}\left(  n^{-\zeta}\right)  a_{n}%
{\displaystyle\int_{0}^{Z_{n-k:n}}}
\frac{\left(  \overline{H}^{\left(  1\right)  }(x)\right)  ^{1/2-\zeta}%
}{\overline{H}(x)}\overline{F}(x)dx\leq O_{\mathbb{P}}\left(  1\right)
n^{-\zeta}a_{n}%
{\displaystyle\int_{0}^{Z_{n-k:n}}}
\frac{\overline{F}(x)}{\left(  \overline{H}(x)\right)  ^{1/2+\zeta}}dx.
\]
We show that%
\[
n^{-\zeta}a_{n}%
{\displaystyle\int_{0}^{Z_{n-k:n}}}
\frac{\overline{F}(x)}{\left(  \overline{H}(x)\right)  ^{1/2+\zeta}%
}dx=n^{-\zeta}a_{n}%
{\displaystyle\int_{0}^{h}}
\frac{\overline{F}(x)}{\left(  \overline{H}(x)\right)  ^{1/2+\zeta}%
}dx+o_{\mathbb{P}}\left(  1\right)  .
\]
Indeed, we have%
\begin{align*}
\left\vert
{\displaystyle\int_{0}^{Z_{n-k:n}}}
\frac{\overline{F}(x)}{\left(  \overline{H}(x)\right)  ^{1/2+\zeta}}dx-%
{\displaystyle\int_{0}^{h}}
\frac{\overline{F}(x)}{\left(  \overline{H}(x)\right)  ^{1/2+\zeta}%
}dx\right\vert  &  =\left\vert
{\displaystyle\int_{h}^{Z_{n-k:n}}}
\frac{\overline{F}(x)}{\left(  \overline{H}(x)\right)  ^{1/2+\zeta}%
}dx\right\vert \\
&  =%
{\displaystyle\int_{\min\left(  h,Z_{n-k:n}\right)  }^{\max\left(
h,Z_{n-k:n}\right)  }}
\frac{\overline{F}(x)}{\left(  \overline{H}(x)\right)  ^{1/2+\zeta}}dx.
\end{align*}
By using Potter's inequalities, given in assertion 5 of Proposition B.1.9 in
\cite{deHF06}, we write for $\epsilon>0,$%
\[
n^{-\zeta}a_{n}%
{\displaystyle\int_{\min\left(  h,Z_{n-k:n}\right)  }^{\min\left(
h,Z_{n-k:n}\right)  }}
\frac{\overline{F}(x)}{\left(  \overline{H}(x)\right)  ^{1/2+\zeta}}dx\leq
k^{-\zeta}\left[  x^{-1/\gamma_{1}+(1/2+\zeta)/\gamma\pm\epsilon}\right]
_{\min\left(  1,Z_{n-k:n}/h\right)  }^{\max\left(  1,Z_{n-k:n}/h\right)  }.
\]
On the other hand, combining Corollary 2.2.2 with Potter's inequalities given
in Proposition B.1.9 (5) in \cite{deHF06}, yields that $Z_{n-k:n}%
/h\rightarrow1$ in probability. Therefore, the right-hand side of the previous
inequality tends to zero, as sought. Now, we show that $T_{n1}$ may be
rewritten into%
\begin{equation}
T_{n1}=a_{n}%
{\displaystyle\int_{0}^{h}}
\dfrac{\overline{F}(x)}{\overline{H}(x)}\mathbf{B}_{n}\left(  x\right)
dx+o_{\mathbb{P}}\left(  1\right)  . \label{approxima1}%
\end{equation}
Observe that%
\[
T_{n1}=a_{n}%
{\displaystyle\int_{0}^{h}}
\dfrac{\overline{F}(x)}{\overline{H}(x)}\mathbf{B}_{n}\left(  x\right)
dx+a_{n}%
{\displaystyle\int_{h}^{Z_{n-k:n}}}
\dfrac{\overline{F}(x)}{\overline{H}(x)}\mathbf{B}_{n}\left(  x\right)
dx+o_{\mathbb{P}}\left(  1\right)  ,
\]
with the second term in the right-hand side tending to zero in probability.
Indeed, for fixed $0<\eta,\epsilon<1,$ we have%
\begin{align*}
&  \mathbb{P}\left(  \left\vert a_{n}%
{\displaystyle\int_{h}^{Z_{n-k:n}}}
\frac{\overline{F}(v)}{\overline{H}(v)}\mathbf{B}_{n}\left(  v\right)
dv\right\vert >\eta\right) \\
&  \leq\mathbb{P}\left(  \left\vert \frac{Z_{n-k:n}}{h}-1\right\vert
>\epsilon\right)  +\mathbb{P}\left(  \left\vert a_{n}\int_{h}^{\left(
1+\epsilon\right)  h}\frac{\overline{F}(v)}{\overline{H}(v)}\mathbf{B}%
_{n}\left(  v\right)  dv\right\vert >\eta\right)  ,
\end{align*}
where, in virtue of the fact that $Z_{n-k:n}/h\overset{\mathbb{P}}%
{\rightarrow}1,$ the first term tends to zero. It remains to show that the
second term in the right-hand side is also asymptotically negligible. We have
$\overline{H}^{\left(  1\right)  }\leq\overline{H},$ then%
\begin{align*}
\mathbf{E}\left\vert a_{n}\int_{h}^{\left(  1+\epsilon\right)  h}%
\frac{\overline{F}(v)}{\overline{H}(v)}\mathbf{B}_{n}\left(  v\right)
dv\right\vert  &  \leq a_{n}\int_{h}^{\left(  1+\epsilon\right)  h}%
\frac{\overline{F}(v)}{\overline{H}(v)}\sqrt{\overline{H}^{\left(  1\right)
}\left(  v\right)  }dv\\
&  \leq a_{n}\int_{h}^{\left(  1+\epsilon\right)  h}\frac{\overline{F}%
(v)}{\sqrt{\overline{H}(v)}}dv.
\end{align*}
Changing variables and applying Potter's inequalities to the regularly varying
function $\overline{F}(x)/\sqrt{\overline{H}(x)},$ yield that, for all large
$n$ and $\xi>0,$ we have%
\begin{align*}
\mathbf{E}\left\vert a_{n}\int_{h}^{\left(  1+\epsilon\right)  h}%
\frac{\overline{F}(v)}{\overline{H}(v)}\mathbf{B}_{n}\left(  v\right)
dv\right\vert  &  \leq a_{n}\frac{h\overline{F}(h)}{\sqrt{\overline{H}(h)}%
}\int_{1}^{1+\epsilon}v^{-1/\gamma_{1}+1/\left(  2\gamma\right)  \pm\xi}dv\\
&  =\int_{1}^{1+\epsilon}v^{-1/\gamma_{1}+1/\left(  2\gamma\right)  \pm\xi}dv.
\end{align*}
The latter integral is clearly finite and tends to zero as $\epsilon
\downarrow0.$ By similar arguments using approximations $\left(
\ref{Cs}\right)  $ and $\left(  \ref{Necir}\right)  ,$ we also show that%
\begin{equation}
T_{n2}=a_{n}%
{\displaystyle\int_{0}^{h}}
\left\{
{\displaystyle\int_{0}^{x}}
\dfrac{\mathbf{B}_{n}(v)}{\overline{H}^{2}(v)}d\overline{H}(v)\right\}
\overline{F}(x)dx+o_{\mathbb{P}}\left(  1\right)  \label{approxima2}%
\end{equation}
and%
\begin{equation}
T_{n3}=-a_{n}%
{\displaystyle\int_{0}^{h}}
\left\{
{\displaystyle\int_{0}^{x}}
\dfrac{\mathbf{B}_{n}^{\ast}(v)}{\overline{H}^{2}(v)}d\overline{H}^{\left(
1\right)  }(v)\right\}  \overline{F}(x)dx+o_{\mathbb{P}}\left(  1\right)  ,
\label{approxima3}%
\end{equation}
where%
\begin{equation}
\mathbf{B}_{n}^{\ast}\left(  x\right)  :=\mathbf{B}_{n}(x)-B_{n}\left(
1-\overline{H}^{\left(  0\right)  }\left(  x\right)  \right)  ,\text{ for
}0<\overline{H}^{\left(  0\right)  }\left(  x\right)  <1-\theta. \label{Bn*}%
\end{equation}
Before we examine $T_{n4},$ we provide an approximation to $T_{n5},$ for which
a change of variables yields%
\[
T_{n5}=-\sqrt{k}%
{\displaystyle\int_{Z_{n-k:n}/h}^{1}}
\frac{\overline{F}(hx)}{\overline{F}(h)}dx.
\]
For the purpose of using the second-order condition of regular variation
$\left(  \ref{Condi}\right)  $ of $\overline{F},$ we write%
\begin{equation}
T_{n5}=-\sqrt{k}A_{1}(h)%
{\displaystyle\int_{Z_{n-k:n}/h}^{1}}
\left(  \dfrac{\overline{F}(hx)/\overline{F}(h)-x^{-1/\gamma_{1}}}{A_{1}%
(h)}\right)  dx-\sqrt{k}%
{\displaystyle\int_{Z_{n-k:n}/h}^{1}}
x^{-1/\gamma_{1}}dx. \label{Tn5}%
\end{equation}
From the inequality $\left(  2.3.23\right)  $ of Theorem 2.3.9 in
\cite{deHF06}, page 48, we infer that the first integral in $\left(
\ref{Tn5}\right)  $ is equal to%
\[
\left(  1+o_{\mathbb{P}}\left(  1\right)  \right)
{\displaystyle\int_{Z_{n-k:n}/h}^{1}}
x^{-1/\gamma_{1}}\left(  x^{\tau_{1}/\gamma_{1}}-1\right)  /\gamma_{1}\tau
_{1}dx,
\]
which tends to zero in probability due to the fact that $Z_{n-k:n}%
/h\overset{\mathbb{P}}{\rightarrow}1.$ Moreover, the term $\sqrt{k}A_{1}(h)$
has, by assumption, a finite limit. Consequently, the first term in the
right-hand side of $\left(  \ref{Tn5}\right)  $ is asymptotically negligible.
We develop the second integral and make a Taylor's expansion. Knowing, once
again, that $Z_{n-k:n}/h\overset{\mathbb{P}}{\rightarrow}1$ ultimately yields
that%
\[
T_{n5}=\left(  1+o_{\mathbb{P}}\left(  1\right)  \right)  \sqrt{k}\left(
\frac{Z_{n-k:n}}{h}-1\right)  .
\]
By using result $\left(  2.7\right)  $ of Theorem 2.1 in \cite{BMN15}, we get%
\begin{equation}
T_{n5}=\gamma\sqrt{\frac{n}{k}}\mathbf{B}_{n}^{\ast}\left(  h\right)
+o_{\mathbb{P}}\left(  1\right)  . \label{approxima4}%
\end{equation}
Next, we readily check that the fourth term $T_{n4}$\textbf{ }tends to zero in
probability. Indeed, we have $%
{\displaystyle\int_{0}^{Z_{n-k:n}}}
\overline{F}(x)dx<\mu$ and by assumption $\sqrt{k}h\overline{F}(h)\rightarrow
\infty.$ Finally, for the last term $T_{n6}$ we use the second-order regular
variation of the tails $\overline{F}$ and $\overline{G}.$ From Lemma 3 in
\cite{Hua}, there exist two positive constants $c_{1}$ $c_{2}$ such that
$h=\left(  1+o\left(  1\right)  \right)  c_{1}\left(  k/n\right)  ^{-\gamma}$
and $\overline{F}(h)=\left(  1+o\left(  1\right)  \right)  c_{2}\left(
k/n\right)  ^{\gamma/\gamma_{1}},$ thus $a_{n}=\left(  1+o\left(  1\right)
\right)  c_{1}c_{2}\left(  k/n\right)  ^{1/2+\gamma-\gamma/\gamma_{1}}.$ But
the indices $\gamma_{1}$ and $\gamma_{2}$ belong to $\mathcal{R},$ hence
$1/2+\gamma-\gamma/\gamma_{1}$ are positive. Therefore, $a_{n}\rightarrow0$
and $T_{n6}=o_{\mathbb{P}}\left(  1\right)  .\mathbb{\ }$The four
approximations $\left(  \ref{approxima1}\right)  ,$ $\left(  \ref{approxima2}%
\right)  ,$ $\left(  \ref{approxima3}\right)  $ and $\left(  \ref{approxima4}%
\right)  $ together with the asymptotic negligibility of both $T_{n4}$ and
$T_{n6}$ give%
\begin{align}
&  \frac{\sqrt{k}\left(  \widehat{\mu}_{1}-\mu_{1}\right)  }{h\overline{F}%
(h)}\nonumber\\
&  =a_{n}%
{\displaystyle\int_{0}^{h}}
\frac{\mathbf{B}_{n}\left(  x\right)  }{\overline{H}(x)}\overline
{F}(x)dx+a_{n}%
{\displaystyle\int_{0}^{h}}
\left\{
{\displaystyle\int_{0}^{x}}
\dfrac{\mathbf{B}_{n}(v)}{\overline{H}^{2}(v)}d\overline{H}(v)\right\}
\overline{F}(x)dx\label{mu-chap1}\\
&  -a_{n}%
{\displaystyle\int_{0}^{h}}
\left\{
{\displaystyle\int_{0}^{x}}
\dfrac{\mathbf{B}_{n}^{\ast}(v)}{\overline{H}^{2}(v)}d\overline{H}^{\left(
1\right)  }(v)\right\}  \overline{F}(x)dx+\gamma\sqrt{\frac{n}{k}}%
\mathbf{B}_{n}^{\ast}\left(  h\right)  +o_{\mathbb{P}}\left(  1\right)
.\nonumber
\end{align}
Let us now treat the term $\sqrt{k}\left(  \widehat{\mu}_{2}-\mu_{2}\right)
/\left(  h\overline{F}(h)\right)  \mathbf{.}$ Consider the following forms of
$\mu_{2}$ and $\widehat{\mu}_{2}:$%

\[
\mu_{2}=h\overline{F}\left(  h\right)  \int_{1}^{\infty}\frac{\overline
{F}\left(  hx\right)  }{\overline{F}\left(  h\right)  }dx\text{ and }%
\widehat{\mu}_{2}=\frac{\widehat{\gamma}_{1}^{(H,c)}}{1-\widehat{\gamma}%
_{1}^{(H,c)}}Z_{n-k:n}\overline{F}\left(  Z_{n-k:n}\right)  \frac{\overline
{F}_{n}\left(  Z_{n-k:n}\right)  }{\overline{F}\left(  Z_{n-k:n}\right)  },
\]
and decompose $\sqrt{k}\left(  \widehat{\mu}_{2}-\mu_{2}\right)  /\left(
h\overline{F}\left(  h\right)  \right)  $ into the sum of%
\[%
\begin{array}
[c]{cl}%
S_{n1}:= & \sqrt{k}\dfrac{\widehat{\gamma}_{1}^{(H,c)}}{1-\widehat{\gamma}%
_{1}^{(H,c)}}\dfrac{\overline{F}\left(  Z_{n-k:n}\right)  }{\overline
{F}\left(  h\right)  }\dfrac{\overline{F}_{n}\left(  Z_{n-k:n}\right)
}{\overline{F}\left(  Z_{n-k:n}\right)  }\left\{  \dfrac{Z_{n-k:n}}%
{h}-1\right\}  ,\bigskip\\
S_{n2}:= & \sqrt{k}\dfrac{\overline{F}\left(  Z_{n-k:n}\right)  }{\overline
{F}\left(  h\right)  }\dfrac{\overline{F}_{n}\left(  Z_{n-k:n}\right)
}{\overline{F}\left(  Z_{n-k:n}\right)  }\left\{  \dfrac{\widehat{\gamma}%
_{1}^{(H,c)}}{1-\widehat{\gamma}_{1}^{(H,c)}}-\dfrac{\gamma_{1}}{1-\gamma_{1}%
}\right\}  ,\bigskip\\
S_{n3}:= & \sqrt{k}\dfrac{\gamma_{1}}{1-\gamma_{1}}\dfrac{\overline{F}\left(
Z_{n-k:n}\right)  }{\overline{F}\left(  h\right)  }\left\{  \dfrac
{\overline{F}_{n}\left(  Z_{n-k:n}\right)  }{\overline{F}\left(
Z_{n-k:n}\right)  }-1\right\}  ,\bigskip\\
S_{n4}:= & \sqrt{k}\dfrac{\gamma_{1}}{1-\gamma_{1}}\left\{  \dfrac
{\overline{F}\left(  Z_{n-k:n}\right)  }{\overline{F}\left(  h\right)
}-\left(  \dfrac{Z_{n-k:n}}{h}\right)  ^{-1/\gamma_{1}}\right\}  ,\bigskip\\
S_{n5}:= & \sqrt{k}\dfrac{\gamma_{1}}{1-\gamma_{1}}\left\{  \left(
\dfrac{Z_{n-k:n}}{h}\right)  ^{-1/\gamma_{1}}-1\right\}  ,\bigskip\\
S_{n6}:= & \sqrt{k}\left\{  \dfrac{\gamma_{1}}{1-\gamma_{1}}-%
{\displaystyle\int_{1}^{\infty}}
\dfrac{\overline{F}\left(  hx\right)  }{\overline{F}\left(  h\right)
}dx\right\}  .
\end{array}
\]
For the first term, we have $\widehat{\gamma}_{1}\overset{\mathbb{P}%
}{\rightarrow}\gamma_{1}$ and $Z_{n-k:n}/h\overset{\mathbb{P}}{\rightarrow}1,$
which, in view of the regular variation of $\overline{F},$ implies that
$\overline{F}\left(  Z_{n-k:n}\right)  =\left(  1+o_{\mathbb{P}}\left(
1\right)  \right)  \overline{F}\left(  h\right)  .$ Moreover, from $\left(
\ref{p(1-p)}\right)  $ we infer that\textbf{ }$\overline{F}_{n}\left(
Z_{n-k:n}\right)  =\left(  1+o_{\mathbb{P}}\left(  1\right)  \right)
\overline{F}\left(  Z_{n-k:n}\right)  .$ It follows that%
\[
S_{n1}=\left(  1+o_{\mathbb{P}}\left(  1\right)  \right)  \frac{\gamma_{1}%
}{1-\gamma_{1}}\sqrt{k}\left(  \dfrac{Z_{n-k:n}}{h}-1\right)  ,
\]
which, by applying result $\left(  2.7\right)  $ of Theorem 2.1 in
\cite{BMN15}, is approximated as follows:%
\begin{equation}
S_{n1}=\left(  1+o_{\mathbb{P}}\left(  1\right)  \right)  \frac{\gamma
_{1}\gamma}{1-\gamma_{1}}\sqrt{\frac{n}{k}}\mathbf{B}_{n}^{\ast}\left(
h\right)  . \label{Sn1}%
\end{equation}
By using similar arguments, we easily show that%
\[
S_{n2}=\left(  1+o_{\mathbb{P}}\left(  1\right)  \right)  \frac{1}{\left(
1-\gamma_{1}\right)  ^{2}}\sqrt{k}\left(  \widehat{\gamma}_{1}^{(H,c)}%
-\gamma_{1}\right)  ,
\]
which, by applying result $\left(  2.9\right)  $ (after a change of variables)
of Theorem 2.1 in \cite{BMN15}, becomes%
\begin{equation}
S_{n2}=\frac{1+o_{\mathbb{P}}\left(  1\right)  }{\left(  1-\gamma_{1}\right)
^{2}}\left(  \frac{1}{p}\sqrt{\frac{n}{k}}\int_{1}^{\infty}v^{-1}%
\mathbf{B}_{n}^{\ast}\left(  hv\right)  dv-\frac{\gamma_{1}}{p}\sqrt{\frac
{n}{k}}\mathbf{B}_{n}\left(  h\right)  +\frac{\sqrt{k}A_{1}\left(  h\right)
}{1-p\tau_{1}}\right)  . \label{Sn2}%
\end{equation}
For $S_{n3},$ we have%
\[
S_{n3}=\left(  1+o_{\mathbb{P}}\left(  1\right)  \right)  \frac{\gamma_{1}%
}{1-\gamma_{1}}\sqrt{k}\left(  \frac{\overline{F}_{n}\left(  Z_{n-k:n}\right)
}{\overline{F}\left(  Z_{n-k:n}\right)  }-1\right)  .
\]
Using Proposition \ref{Prop1}, we have%
\begin{align}
S_{n3}  &  =\left(  1+o_{\mathbb{P}}\left(  1\right)  \right)  \sqrt{\frac
{k}{n}}\frac{\gamma_{1}}{1-\gamma_{1}}\left(
{\displaystyle\int_{0}^{h}}
\dfrac{\mathbf{B}_{n}(v)}{\overline{H}^{2}(v)}d\overline{H}(v)-%
{\displaystyle\int_{0}^{h}}
\dfrac{\mathbf{B}_{n}^{\ast}(v)}{\overline{H}^{2}(v)}d\overline{H}^{\left(
1\right)  }(v)\right) \label{Sn3}\\
&  +\left(  1+o_{\mathbb{P}}\left(  1\right)  \right)  \frac{\gamma_{1}%
}{1-\gamma_{1}}\sqrt{\frac{n}{k}}\mathbf{B}_{n}\left(  h\right)
+o_{\mathbb{P}}\left(  1\right)  .\nonumber
\end{align}
For the fourth term, we use the second-order condition $\left(  \ref{Condi}%
\right)  $ of $\overline{F}$ and the fact that $Z_{n-k:n}/h\overset
{\mathbb{P}}{\rightarrow}1$ to get%
\begin{equation}
S_{n4}=o_{\mathbb{P}}\left(  \sqrt{k}A_{1}\left(  h\right)  \right)
=o_{\mathbb{P}}\left(  1\right)  ,\text{ as }n\rightarrow\infty. \label{Sn4}%
\end{equation}
For $S_{n5},$ we apply the mean value theorem with the fact $Z_{n-k:n}%
/h\overset{\mathbb{P}}{\rightarrow}1$ to have%
\[
S_{n5}=-\left(  1+o_{\mathbb{P}}\left(  1\right)  \right)  \frac{1}%
{1-\gamma_{1}}\sqrt{k}\left(  \frac{Z_{n-k:n}}{h}-1\right)  .
\]
Using, once again, result $\left(  2.7\right)  $ of Theorem 2.1 in
\cite{BMN15} yields%
\begin{equation}
S_{n5}=-\left(  1+o_{\mathbb{P}}\left(  1\right)  \right)  \frac{\gamma
}{1-\gamma_{1}}\sqrt{\frac{n}{k}}\mathbf{B}_{n}^{\ast}\left(  h\right)  .
\label{Sn5}%
\end{equation}
For the last term, we first note that%
\[
\frac{S_{n6}}{\sqrt{k}}=\int_{1}^{\infty}x^{-1/\gamma_{1}}dx-\int_{1}^{\infty
}\frac{\overline{F}\left(  hx\right)  }{\overline{F}\left(  h\right)  }dx.
\]
Then, by applying the uniform inequality of regularly varying functions
\citep[see, e.g., Theorem 2.3.9 in][page 48]{deHF06} together with the regular
variation of $\left\vert A_{1}\right\vert ,$ we show that%
\begin{equation}
S_{n6}\sim\frac{\sqrt{k}A_{1}\left(  h\right)  }{\left(  \gamma_{1}+\tau
_{1}-1\right)  \left(  1-\gamma_{1}\right)  }. \label{Sn6}%
\end{equation}
By gathering $\left(  \ref{Sn1}\right)  ,$ $\left(  \ref{Sn2}\right)  ,$
$\left(  \ref{Sn3}\right)  ,$ $\left(  \ref{Sn4}\right)  ,$ $\left(
\ref{Sn5}\right)  $ and $\left(  \ref{Sn6}\right)  $ we end up with%
\begin{align}
\dfrac{\sqrt{k}\left(  \widehat{\mu}_{2}-\mu_{2}\right)  }{h\overline{F}(h)}
&  =\dfrac{\gamma_{1}}{1-\gamma_{1}}\sqrt{\dfrac{k}{n}}\left\{
{\displaystyle\int_{0}^{h}}
\dfrac{\mathbf{B}_{n}(v)}{\overline{H}^{2}(v)}d\overline{H}(v)-%
{\displaystyle\int_{0}^{h}}
\dfrac{\mathbf{B}_{n}^{\ast}(v)}{\overline{H}^{2}(v)}d\overline{H}^{\left(
1\right)  }(v)\right\} \nonumber\\
&  +\sqrt{\dfrac{n}{k}}\left\{  -\dfrac{\gamma_{1}\mathbf{B}_{n}\left(
h\right)  }{p\left(  1-\gamma_{1}\right)  ^{2}}-\gamma\mathbf{B}_{n}^{\ast
}\left(  h\right)  +\dfrac{%
{\displaystyle\int_{1}^{\infty}}
v^{-1}\mathbf{B}_{n}^{\ast}\left(  hv\right)  dv}{p\left(  1-\gamma
_{1}\right)  ^{2}}\right\} \label{mu-chap2}\\
&  +R_{n1}+o_{\mathbb{P}}(1),\nonumber
\end{align}
where%
\[
R_{n1}:=\dfrac{\sqrt{k}A_{1}\left(  h\right)  }{\left(  1-\gamma_{1}\right)
}\left\{  \dfrac{1}{\left(  1-p\tau_{1}\right)  \left(  1-\gamma_{1}\right)
}+\dfrac{1}{\left(  \gamma_{1}+\tau_{1}-1\right)  }\right\}  .
\]
Finally, by summing up equations $\left(  \ref{mu-chap1}\right)  $ and
$\left(  \ref{mu-chap2}\right)  $ we obtain%
\[
\frac{\sqrt{k}\left(  \widehat{\mu}-\mu\right)  }{h\overline{F}(h)}=\sum
_{i=1}^{5}D_{ni}+R_{n1}+o_{\mathbb{P}}(1),
\]
where%
\[%
\begin{array}
[c]{cl}%
D_{n1}:= & a_{n}%
{\displaystyle\int_{0}^{h}}
\dfrac{\mathbf{B}_{n}\left(  v\right)  }{\overline{H}(v)}\overline
{F}(v)dv,\text{ \ }D_{n2}:=a_{n}%
{\displaystyle\int_{0}^{h}}
\left\{
{\displaystyle\int_{0}^{x}}
\dfrac{\mathbf{B}_{n}(v)}{\overline{H}^{2}(v)}d\overline{H}(v)\right\}
\overline{F}(x)dx,\bigskip\\
D_{n3}:= & -a_{n}%
{\displaystyle\int_{0}^{h}}
\left\{
{\displaystyle\int_{0}^{x}}
\dfrac{\mathbf{B}_{n}^{\ast}(v)}{\overline{H}^{2}(v)}d\overline{H}^{\left(
1\right)  }(v)\right\}  \overline{F}(x)dx,\bigskip\\
D_{n4}:= & \dfrac{\gamma_{1}}{1-\gamma_{1}}\sqrt{\dfrac{k}{n}}\left(
{\displaystyle\int_{0}^{h}}
\dfrac{\mathbf{B}_{n}(v)}{\overline{H}^{2}(v)}d\overline{H}(v)-%
{\displaystyle\int_{0}^{h}}
\dfrac{\mathbf{B}_{n}^{\ast}(v)}{\overline{H}^{2}(v)}d\overline{H}^{\left(
1\right)  }(v)\right)  ,\bigskip\\
D_{n5}:= & \sqrt{\dfrac{n}{k}}\left(  -\dfrac{\gamma_{1}}{p\left(
1-\gamma_{1}\right)  ^{2}}\mathbf{B}_{n}\left(  h\right)  +\dfrac{1}{p\left(
1-\gamma_{1}\right)  ^{2}}%
{\displaystyle\int_{1}^{\infty}}
v^{-1}\mathbf{B}_{n}^{\ast}\left(  hv\right)  dv\right)  .
\end{array}
\]
Note that $D_{n2}$ is of the form $-a_{n}%
{\displaystyle\int_{0}^{h}}
\psi\left(  x\right)  d\varphi\left(  x\right)  ,$ where $\varphi\left(
x\right)  :=%
{\displaystyle\int_{x}^{\infty}}
\overline{F}(u)du$ and $\psi\left(  x\right)  :=%
{\displaystyle\int_{0}^{x}}
\mathbf{B}_{n}(v)/\overline{H}^{2}(v)d\overline{H}(v).$ Integrating by parts
yields%
\[
D_{n2}=a_{n}%
{\displaystyle\int_{0}^{h}}
\varphi\left(  v\right)  \dfrac{\mathbf{B}_{n}(v)}{\overline{H}^{2}%
(v)}d\overline{H}(v)-\sqrt{\frac{k}{n}}\frac{%
{\displaystyle\int_{h}^{\infty}}
\overline{F}(x)dx}{h\overline{F}(h)}%
{\displaystyle\int_{0}^{h}}
\dfrac{\mathbf{B}_{n}(v)}{\overline{H}^{2}(v)}d\overline{H}(v)
\]
\textbf{ }Equation $\left(  B.1.9\right)  $ in Theorem B.1.5 (Karamata's
theorem) in \cite{deHF06} yields that $%
{\displaystyle\int_{h}^{\infty}}
\overline{F}(x)dx/\left(  h\overline{F}(h)\right)  \rightarrow\gamma
_{1}/\left(  1-\gamma_{1}\right)  .$ We apply the same technique to $D_{n3}$
and get%
\[
D_{n2}+D_{n3}+D_{n4}=L_{n2}+L_{n3}+R_{n2},
\]
where $R_{n2}:=o_{\mathbb{P}}(D_{n4})$ and%
\[
L_{n2}:=a_{n}%
{\displaystyle\int_{0}^{h}}
\dfrac{\mathbf{B}_{n}(v)}{\overline{H}^{2}(v)}\varphi\left(  v\right)
d\overline{H}(v)\text{ and }L_{n3}:=-a_{n}%
{\displaystyle\int_{0}^{h}}
\dfrac{\mathbf{B}_{n}^{\ast}(v)}{\overline{H}^{2}(v)}\varphi\left(  v\right)
d\overline{H}^{\left(  1\right)  }(v).
\]
This yields the following new decomposition:%

\[
\frac{\sqrt{k}\left(  \widehat{\mu}-\mu\right)  }{h\overline{F}(h)}=\sum
_{i=1}^{4}L_{ni}+R_{n1}+R_{n2}+o_{\mathbb{P}}(1),
\]
with $L_{n1}:=D_{n1}$ and $L_{n4}:=D_{n5}.$ The four $L_{ni}$ are centred
Gaussian rv's whose asymptotic second moments are finite, as we will see
thereafter. Indeed, $L_{n4}$ is the Gaussian approximation to Hill's estimator
given by result $\left(  2.9\right)  $ of Theorem 2.1 in \cite{BMN15}, hence
we have $\lim_{n\rightarrow\infty}\mathbf{E}\left[  L_{n4}^{2}\right]
<\infty.$ For the three others, we literally compute the asymptotic moments of
order two. Note that from the covariance structure in \cite{C96}, page 2768,
we have the following useful formulas:%
\begin{equation}
\left\{
\begin{tabular}
[c]{l}%
$\mathbf{E}\left[  \mathbf{B}_{n}\left(  u\right)  \mathbf{B}_{n}\left(
v\right)  \right]  =\min\left(  \overline{H}^{\left(  1\right)  }\left(
u\right)  ,\overline{H}^{\left(  1\right)  }\left(  v\right)  \right)
-\overline{H}^{\left(  1\right)  }\left(  u\right)  \overline{H}^{\left(
1\right)  }\left(  v\right)  ,\smallskip$\\
$\mathbf{E}\left[  \mathbf{B}_{n}^{\ast}\left(  u\right)  \mathbf{B}_{n}%
^{\ast}\left(  v\right)  \right]  =\min\left(  \overline{H}\left(  u\right)
,\overline{H}\left(  v\right)  \right)  -\overline{H}\left(  u\right)
\overline{H}\left(  v\right)  ,\smallskip$\\
$\mathbf{E}\left[  \mathbf{B}_{n}\left(  u\right)  \mathbf{B}_{n}^{\ast
}\left(  v\right)  \right]  =\min\left(  \overline{H}^{\left(  1\right)
}\left(  u\right)  ,\overline{H}^{\left(  1\right)  }\left(  v\right)
\right)  -\overline{H}^{\left(  1\right)  }\left(  u\right)  \overline
{H}\left(  v\right)  .$%
\end{tabular}
\ \right.  \label{covariances}%
\end{equation}
After elementary but very tedious calculations, using these formulas with
l'H\^{o}pital's rule, we get as $n\rightarrow\infty,$%
\begin{equation}
\left\{
\begin{array}
[c]{l}%
\dfrac{k}{n}%
{\displaystyle\int_{0}^{h}}
{\displaystyle\int_{0}^{h}}
\dfrac{\mathbf{E}\left[  \mathbf{B}_{n}\left(  u\right)  \mathbf{B}_{n}\left(
v\right)  \right]  }{\overline{H}^{2}\left(  u\right)  \overline{H}^{2}\left(
v\right)  }d\overline{H}\left(  u\right)  d\overline{H}\left(  v\right)
\rightarrow p,\bigskip\\
\dfrac{k}{n}%
{\displaystyle\int_{0}^{h}}
{\displaystyle\int_{0}^{h}}
\dfrac{\mathbf{E}\left[  \mathbf{B}_{n}^{\ast}\left(  u\right)  \mathbf{B}%
_{n}^{\ast}\left(  v\right)  \right]  }{\overline{H}^{2}\left(  u\right)
\overline{H}^{2}\left(  v\right)  }d\overline{H}^{\left(  1\right)  }\left(
u\right)  d\overline{H}^{\left(  1\right)  }\left(  v\right)  \rightarrow
p^{2},\bigskip\\
\dfrac{k}{n}%
{\displaystyle\int_{0}^{h}}
{\displaystyle\int_{0}^{h}}
\dfrac{\mathbf{E}\left[  \mathbf{B}_{n}\left(  u\right)  \mathbf{B}_{n}^{\ast
}\left(  v\right)  \right]  }{\overline{H}^{2}\left(  u\right)  \overline
{H}^{2}\left(  v\right)  }d\overline{H}\left(  u\right)  d\overline
{H}^{\left(  1\right)  }\left(  v\right)  \rightarrow p^{2}.
\end{array}
\right.  \label{limits}%
\end{equation}
By using the results above, we obtain%
\[
\mathbf{E}\left[  L_{n1}\right]  ^{2}\rightarrow\frac{2\gamma^{2}\gamma
_{1}^{2}}{\left(  \gamma_{1}-\gamma+\gamma\gamma_{1}\right)  \left(
\gamma_{1}-2\gamma+2\gamma\gamma_{1}\right)  },
\]%
\[
\mathbf{E}\left[  L_{n2}\right]  ^{2}\rightarrow\frac{2p\gamma_{1}^{4}%
}{\left(  \gamma_{1}-1\right)  ^{2}\left(  \gamma_{1}-\gamma+\gamma\gamma
_{1}\right)  \left(  \gamma_{1}-2\gamma+2\gamma\gamma_{1}\right)  },
\]
and%
\[
\mathbf{E}\left[  L_{n3}\right]  ^{2}\rightarrow\frac{2p^{2}\gamma_{1}^{4}%
}{\left(  \gamma_{1}-1\right)  ^{2}\left(  \gamma_{1}-\gamma+\gamma\gamma
_{1}\right)  \left(  \gamma_{1}-2\gamma+2\gamma\gamma_{1}\right)  }.
\]
As a consequence, we conclude that%
\[
\sqrt{k}\frac{\widehat{\mu}-\mu}{h\overline{F}(h)}\overset{d}{\rightarrow
}\mathcal{N}\left(  m,\sigma^{2}\right)  ,\text{ as }n\rightarrow\infty,
\]
where $m:=\lim_{n\rightarrow\infty}R_{n1}$ and $\sigma^{2}:=\lim
_{n\rightarrow\infty}\mathbf{E}\left[  \sum_{i=1}^{4}L_{ni}\right]  ^{2}%
.$\ The expression of $m$ is simple and easily obtainable whilst that of
$\sigma^{2}$ is very complicated and requires extremely laborious
computations. However, we readily check that, it is finite. Indeed, in
addition to the finiteness of the asymptotic second moments $\mathbf{E}\left[
L_{ni}\right]  ^{2}\mathbf{,}$ the asymptotic covariances $\mathbf{E}\left[
L_{ni}L_{nj}\right]  $ are, in virtue of Cauchy-Schwarz inequality, finite as
well. Finally, we use the facts that $Z_{n-k:n}/h$ and $\overline{F}%
(Z_{n-k:n})/\overline{F}(h)$ tend to $1$ in probability to achieve the
proof.\hfill$\Box$

\section{\textbf{Appendix}}

\noindent In the following basic proposition, we give an asymptotic
representation to the Kaplan-Meier product limit estimator (in $Z_{n-k:n}).$
This result will of prime importance in the study of the limiting behaviors of
many statistics based on censored data exhibiting extreme values.

\begin{proposition}
\textbf{\label{Prop1}}Assume that the second-order conditions $\left(
\ref{Condi}\right)  $ hold. Let $k=k_{n}$ be an integer sequence satisfying,
in addition to $(\ref{k}),$ $\sqrt{k}A_{j}\left(  h\right)  =O\left(
1\right)  ,$ for $j=1,2,$ as $n\rightarrow\infty.$ Then there exists a
sequence of Brownian bridges $\left\{  B_{n}\left(  s\right)  ;\text{ }0\leq
s\leq1\right\}  $ such that%
\begin{align*}
\sqrt{k}\left(  \frac{\overline{F}_{n}\left(  Z_{n-k:n}\right)  }{\overline
{F}\left(  Z_{n-k:n}\right)  }-1\right)   &  =\sqrt{\frac{n}{k}}\mathbf{B}%
_{n}\left(  h\right) \\
&  +\sqrt{\frac{k}{n}}\left(  \int_{0}^{h}\frac{\mathbf{B}_{n}\left(
v\right)  }{\overline{H}^{2}\left(  v\right)  }d\overline{H}\left(  v\right)
-\int_{0}^{h}\frac{\mathbf{B}_{n}^{\ast}\left(  v\right)  }{\overline{H}%
^{2}\left(  v\right)  }d\overline{H}^{\left(  1\right)  }\left(  v\right)
\right)  +o_{\mathbb{P}}\left(  1\right)  ,
\end{align*}
where $\mathbf{B}_{n}\left(  v\right)  $ and $\mathbf{B}_{n}^{\ast}\left(
v\right)  $ are respectively defined in $\left(  \ref{Bn}\right)  $ and
$\left(  \ref{Bn*}\right)  .$ Consequently,%
\begin{equation}
\sqrt{k}\left(  \frac{\overline{F}_{n}\left(  Z_{n-k:n}\right)  }{\overline
{F}\left(  Z_{n-k:n}\right)  }-1\right)  \overset{d}{\rightarrow}%
\mathcal{N}\left(  0,p\right)  ,\text{ as }n\rightarrow\infty. \label{p(1-p)}%
\end{equation}

\end{proposition}

\begin{proof}
Multiplying $\left(  \ref{ratio}\right)  $ by $\sqrt{k}$ yields%
\begin{align*}
&  \sqrt{k}\frac{\overline{F}_{n}\left(  x\right)  -\overline{F}\left(
x\right)  }{\overline{F}\left(  x\right)  }=\sqrt{\frac{k}{n}}\frac{\alpha
_{n}\left(  \theta\right)  -\alpha_{n}\left(  \theta-\overline{H}^{\left(
1\right)  }\left(  x\right)  \right)  }{\overline{H}\left(  x\right)  }\\
&  +\sqrt{\frac{k}{n}}\int_{0}^{x}\frac{\alpha_{n}\left(  \theta\right)
-\alpha_{n}\left(  \theta-\overline{H}^{\left(  1\right)  }\left(  v\right)
\right)  }{\overline{H}^{2}\left(  v\right)  }d\overline{H}\left(  v\right) \\
&  -\sqrt{\frac{k}{n}}\int_{0}^{x}\frac{\alpha_{n}\left(  \theta\right)
-\alpha_{n}\left(  \theta-\overline{H}^{\left(  1\right)  }\left(  v\right)
\right)  -\alpha_{n}\left(  1-\overline{H}^{\left(  0\right)  }\left(
v\right)  \right)  }{\overline{H}^{2}\left(  v\right)  }d\overline{H}^{\left(
1\right)  }\left(  v\right) \\
&  +O_{\mathbb{P}}\left(  \frac{1}{\sqrt{k}}\right)  +O_{\mathbb{P}}\left(
\sqrt{\frac{k}{n}}\right)  .
\end{align*}
The Gaussian approximations $\left(  \ref{Cs}\right)  $ and $\left(
\ref{Necir}\right)  ,$ in $x=Z_{n-k:n},$ and the facts that $\sqrt{k/n}$ and
$1/\sqrt{k}$ tend to zero as $n\rightarrow\infty,$ lead to%
\begin{align*}
&  \sqrt{k}\frac{\overline{F}_{n}\left(  Z_{n-k:n}\right)  -\overline
{F}\left(  Z_{n-k:n}\right)  }{\overline{F}\left(  Z_{n-k:n}\right)  }%
=\sqrt{\frac{n}{k}}\mathbf{B}_{n}\left(  Z_{n-k:n}\right)  +\sqrt{\frac{k}{n}%
}\int_{0}^{Z_{n-k:n}}\frac{\mathbf{B}_{n}\left(  v\right)  }{\overline{H}%
^{2}\left(  v\right)  }d\overline{H}\left(  v\right) \\
&  -\sqrt{\frac{k}{n}}\int_{0}^{Z_{n-k:n}}\frac{\mathbf{B}_{n}^{\ast}\left(
v\right)  }{\overline{H}^{2}\left(  v\right)  }d\overline{H}^{\left(
1\right)  }\left(  v\right)  +o_{\mathbb{P}}\left(  1\right)  .
\end{align*}
Applying Lemma \ref{Lem3} completes the proof. The asymptotic normality
property is straightforward. For the variance computation, we use, in addition
to $\left(  \ref{limits}\right)  ,$ the following results:%
\[
\int_{0}^{h}\frac{\mathbf{E}\left[  \mathbf{B}_{n}\left(  u\right)
\mathbf{B}_{n}\left(  h\right)  \right]  }{\overline{H}^{2}\left(  u\right)
}d\overline{H}\left(  u\right)  \rightarrow-p\text{ and }\int_{0}^{h}%
\frac{\mathbf{E}\left[  \mathbf{B}_{n}^{\ast}\left(  u\right)  \mathbf{B}%
_{n}\left(  h\right)  \right]  }{\overline{H}^{2}\left(  u\right)  }%
d\overline{H}^{1}\left(  u\right)  \rightarrow-p^{2},
\]
similarly obtained as $\left(  \ref{limits}\right)  .$
\end{proof}

\begin{lemma}
\label{Lem3}Assume that the second-order conditions of regular variation
$\left(  \ref{Condi}\right)  $ and let $k:=k_{n}$ be an integer sequence
satisfying $(\ref{k}).$ Then%
\[%
\begin{tabular}
[c]{l}%
$\left(  i\right)  \text{ }\sqrt{\dfrac{k}{n}}%
{\displaystyle\int_{h}^{Z_{n-k:n}}}
\dfrac{\mathbf{B}_{n}\left(  v\right)  }{\overline{H}^{2}\left(  v\right)
}d\overline{H}\left(  v\right)  =o_{\mathbb{P}}\left(  1\right)  .\medskip$\\
$\left(  ii\right)  \text{ }\sqrt{\dfrac{k}{n}}%
{\displaystyle\int_{h}^{Z_{n-k:n}}}
\dfrac{\mathbf{B}_{n}^{\ast}\left(  v\right)  }{\overline{H}^{2}\left(
v\right)  }d\overline{H}^{\left(  1\right)  }\left(  v\right)  =o_{\mathbb{P}%
}\left(  1\right)  .\medskip$\\
$(iii)$ $\sqrt{\dfrac{n}{k}}\left\{  \mathbf{B}_{n}\left(  Z_{n-k:n}\right)
-\mathbf{B}_{n}\left(  h\right)  \right\}  =o_{\mathbb{P}}\left(  1\right)
.\medskip$\\
$(iv)$ $\sqrt{\dfrac{n}{k}}\left\{  \mathbf{B}_{n}^{\ast}\left(
Z_{n-k:n}\right)  -\mathbf{B}_{n}^{\ast}\left(  h\right)  \right\}
=o_{\mathbb{P}}\left(  1\right)  .$%
\end{tabular}
\ \ \ \ \ \ \ \ \
\]

\end{lemma}

\begin{proof}
We begin by proving the first assertion. For fixed $0<\eta,\epsilon<1,$ we
have%
\begin{align*}
&  \mathbb{P}\left(  \left\vert \sqrt{\frac{k}{n}}\int_{h}^{Z_{n-k:n}%
}\mathbf{B}_{n}\left(  v\right)  \frac{d\overline{H}\left(  v\right)
}{\overline{H}^{2}\left(  v\right)  }\right\vert >\eta\right) \\
&  \leq\mathbb{P}\left(  \left\vert \frac{Z_{n-k:n}}{h}-1\right\vert
>\epsilon\right)  +\mathbb{P}\left(  \left\vert \sqrt{\frac{k}{n}}\int
_{h}^{\left(  1+\epsilon\right)  h}\mathbf{B}_{n}\left(  v\right)
\frac{d\overline{H}\left(  v\right)  }{\overline{H}^{2}\left(  v\right)
}\right\vert >\eta\right)  .
\end{align*}
It is clear that the first term in the right-hand side tends to zero as
$n\rightarrow\infty.$ Then, it remains to show that the second one goes to
zero as well.\ Indeed, observe that%
\[
\mathbf{E}\left\vert \sqrt{\frac{k}{n}}\int_{h}^{\left(  1+\epsilon\right)
h}\mathbf{B}_{n}\left(  v\right)  \frac{d\overline{H}\left(  v\right)
}{\overline{H}^{2}\left(  v\right)  }\right\vert \leq-\sqrt{\frac{k}{n}}%
\int_{h}^{\left(  1+\epsilon\right)  h}\mathbf{E}\left\vert \mathbf{B}%
_{n}\left(  v\right)  \right\vert \frac{d\overline{H}\left(  v\right)
}{\overline{H}^{2}\left(  v\right)  }.
\]
From the first result of $\left(  \ref{covariances}\right)  ,$ we have
$\mathbf{E}\left\vert \mathbf{B}_{n}\left(  v\right)  \right\vert \leq
\sqrt{\overline{H}^{\left(  1\right)  }\left(  v\right)  }.$ Then%
\[
\mathbf{E}\left\vert \sqrt{\frac{k}{n}}\int_{h}^{\left(  1+\epsilon\right)
h}\mathbf{B}_{n}\left(  v\right)  \frac{d\overline{H}\left(  v\right)
}{\overline{H}^{2}\left(  v\right)  }\right\vert \leq-\sqrt{\frac{k}{n}}%
\int_{h}^{\left(  1+\epsilon\right)  h}\sqrt{\overline{H}^{\left(  1\right)
}\left(  v\right)  }\frac{d\overline{H}\left(  v\right)  }{\overline{H}%
^{2}\left(  v\right)  },
\]
which, in turn, is less than or equal to%
\[
\sqrt{\frac{k}{n}}\sqrt{\overline{H}^{\left(  1\right)  }\left(  h\right)
}\left(  \frac{1}{\overline{H}\left(  \left(  1+\epsilon\right)  h\right)
}-\frac{1}{\overline{H}\left(  h\right)  }\right)  .
\]
Since $\overline{H}\left(  h\right)  =k/n,$ then this may be rewritten into%
\[
\sqrt{\frac{\overline{H}^{\left(  1\right)  }\left(  h\right)  }{\overline
{H}\left(  h\right)  }}\left(  \frac{\overline{H}\left(  h\right)  }%
{\overline{H}\left(  \left(  1+\epsilon\right)  h\right)  }-1\right)  .
\]
Since $\overline{H}^{\left(  1\right)  }\left(  h\right)  \sim p\overline
{H}\left(  h\right)  $ and $\overline{H}\in\mathcal{RV}_{\left(
-1/\gamma\right)  },$ then the previous quantity tends to $p^{1/2}\left(
\left(  1+\epsilon\right)  ^{1/\gamma}-1\right)  $ as $n\rightarrow\infty.$
Being arbitrary, $\epsilon$ may be chosen small enough so that this limit be
zero.$\ $By similar arguments, we also show assertion $\left(  ii\right)  ,$
therefore we omit the details. The last two assertions are shown following the
same technique, that we use to prove $\left(  iv\right)  .$ Notice that, from
the definition of $\mathbf{B}_{n}^{\ast}\left(  v\right)  $ and the second
covariance formula in $\left(  \ref{covariances}\right)  ,$ we have%
\[
\left\{  \mathbf{B}_{n}^{\ast}\left(  v\right)  ;\text{ }v\geq0\right\}
\overset{d}{=}\left\{  \mathcal{B}_{n}\left(  \overline{H}\left(  v\right)
\right)  ;\text{ }v\geq0\right\}  ,
\]
where $\left\{  \mathcal{B}_{n}\left(  s\right)  ;\text{ }0\leq s\leq
1\right\}  $ is a sequence of standard Brownian bridges. Hence%
\[
\sqrt{\dfrac{n}{k}}\left\{  \mathbf{B}_{n}^{\ast}\left(  Z_{n-k:n}\right)
-\mathbf{B}_{n}^{\ast}\left(  h\right)  \right\}  \overset{d}{=}\sqrt
{\dfrac{n}{k}}\left\{  \mathcal{B}_{n}\left(  \overline{H}\left(
Z_{n-k:n}\right)  \right)  -\mathcal{B}_{n}\left(  \overline{H}\left(
h\right)  \right)  \right\}  .
\]
Let $\left\{  \mathcal{W}_{n}\left(  t\right)  ;\text{ }0\leq s\leq1\right\}
$ be a sequence of standard Wiener processes such that $\mathcal{B}_{n}\left(
t\right)  =\mathcal{W}_{n}\left(  t\right)  -t\mathcal{W}_{n}\left(  1\right)
.$ Then $\sqrt{n/k}\left\{  \mathbf{B}_{n}^{\ast}\left(  Z_{n-k:n}\right)
-\mathbf{B}_{n}^{\ast}\left(  h\right)  \right\}  $ equals in distribution to%
\[
\sqrt{\dfrac{n}{k}}\left(  \left\{  \mathcal{W}_{n}\left(  \overline{H}\left(
Z_{n-k:n}\right)  \right)  -\mathcal{W}_{n}\left(  \overline{H}\left(
h\right)  \right)  \right\}  -\left\{  \overline{H}\left(  Z_{n-k:n}\right)
-\overline{H}\left(  h\right)  \right\}  \mathcal{W}_{n}\left(  1\right)
\right)  .
\]
Using the facts that $\overline{H}\left(  h\right)  =k/n$ and $\overline
{H}\left(  Z_{n-k:n}\right)  /\overline{H}\left(  h\right)  =\left(
1+o_{\mathbb{P}}\left(  1\right)  \right)  ,$ we get%
\[
\sqrt{\dfrac{n}{k}}\left(  \overline{H}\left(  Z_{n-k:n}\right)  -\overline
{H}\left(  h\right)  \right)  =\sqrt{\dfrac{k}{n}}\left(  \frac{\overline
{H}\left(  Z_{n-k:n}\right)  }{\overline{H}\left(  h\right)  }-1\right)
=o_{\mathbb{P}}\left(  1\right)  .
\]
On the other hand $\sqrt{n/k}\left\{  \mathcal{W}_{n}\left(  \overline
{H}\left(  Z_{n-k:n}\right)  \right)  -\mathcal{W}_{n}\left(  \overline
{H}\left(  h\right)  \right)  \right\}  $ is a sequence of Gaussian rv's with
mean zero and variance
\[
\frac{n}{k}\left(  \overline{H}\left(  Z_{n-k:n}\right)  -\overline{H}\left(
h\right)  \right)  =\frac{\overline{H}\left(  Z_{n-k:n}\right)  }{\overline
{H}\left(  h\right)  }-1,
\]
which tends to zero (in probability), as $n\rightarrow\infty.$ This achieves
the proof.
\end{proof}

\end{document}